\RequirePackage{fix-cm}
\documentclass[smallextended]{svjour3}       
\smartqed  
\usepackage{graphicx}
\usepackage{latexsym,amsmath,amssymb,graphicx}
\usepackage{float}
\usepackage[bf, small]{caption}
\usepackage[all,web]{xy}

\usepackage{hyperref}
\def\urlfont{\DeclareFontFamily{OT1}{cmtt}{\hyphenchar\font='057}
              \normalfont\ttfamily \hyphenpenalty=10000}

\DeclareFontFamily{OT1}{rsfs10}{}
\DeclareFontShape{OT1}{rsfs10}{m}{n}{ <-> rsfs10 }{}
\DeclareMathAlphabet{\mathscript}{OT1}{rsfs10}{m}{n}

\DeclareMathOperator{\im}{Im}       
\DeclareMathOperator{\Spec}{Spec}   
\DeclareMathOperator{\Hom}{Hom}     
\DeclareMathOperator{\Tors}{Tors}    
\DeclareMathOperator{\Exc}{Exc}     
\DeclareMathOperator{\Pic}{Pic}     
\DeclareMathOperator{\Cl}{Cl}       
\DeclareMathOperator{\Div}{Div}     
\DeclareMathOperator{\Cox}{Cox}     
\DeclareMathOperator{\rk}{rk}       
\DeclareMathOperator{\Supp}{Supp}   
\DeclareMathOperator{\Mov}{Mov}     
\DeclareMathOperator{\Nef}{Nef}     
\DeclareMathOperator{\Eff}{Eff}     
\DeclareMathOperator{\Relint}{Relint}  
\DeclareMathOperator{\codim}{codim} 
\DeclareMathOperator{\Char}{char}   

\def \wrt{with respect to }

\def \d{\delta }

\def \s{\sigma }

\def \Ga{\Gamma }
\def \Si{\Sigma }
\def \g{\gamma}
\def \vf{\varphi}
\def \ve{\varepsilon}

\def \q{\mathbf{q}}

\def \v{\mathbf{v}}

\def \m{\mathbf{m}}

\def \x{\mathbf{x}}

\def \1{\mathbf{1}}
\def \0{\mathbf{0}}

\def\P{{\mathbb{P}}}

\def\p2{\mathbb{P}^2}
\def\p3{\mathbb{P}^3}
\def\p4{\mathbb{P}^4}

\def\cO{\mathcal{O}}

\def\rk{\operatorname{rk}}

\def\Z{\mathbb{Z}}

\def\C{\mathbb{C}}
\def\K{\mathbb{K}}
\def\R{\mathbb{R}}

\def\Q{\mathbb{Q}}
\def\N{\mathbb{N}}
\def\T{\mathbb{T}}

\def\B{\mathcal{B}}
\def\cD{\mathcal{D}}
\def\X{\frak{X}}
\def\cox{\mathcal{C}ox}
\def\irr{\mathcal{I}rr}
\def\cS{\mathcal{S}}

\def\SF{\mathcal{SF}}

\def\G{\mathcal{G}}

\def\I{\mathcal{I}}
\def\Ls{\mathcal{L}}
\def\Ga{\Gamma}
\def\De{\Delta}

\def\sqm{s\,$\Q$m }

\newtheorem{remarks}{Remarks}

\newcommand{\halfline}{\vskip6pt}
\begin{document}

\title{Embedding non-projective Mori Dream Space}

\author{Michele Rossi}

\thanks{The author was partially supported by I.N.D.A.M. (Istituto Nazionale d'Alta Matematica) as a member of the G.N.S.A.G.A. (Gruppo Nazionale per le Strutture Algebriche, Geometriche e loro Applicazioni)}
\titlerunning{Embedding non-projective Mori Dream Spaces}        

\authorrunning{M. Rossi} 

\institute{M. Rossi \at
              Dipartimento di Matematica, Universit\`a di Torino,
via Carlo Alberto 10, 10123 Torino \\
              Tel.: +39 011 670 2813\\
              Fax: +39 011 670 2878\\
              \email{michele.rossi@unito.it}    }

\date{Received: date / Accepted: date}

\maketitle

\begin{abstract}
This paper is devoted to extend some  Hu-Keel results on Mori dream spaces (MDS) beyond the projective setup. Namely, $\Q$-factorial algebraic varieties with finitely generated class group and Cox ring, here called \emph{weak} Mori dream spaces (wMDS), are considered.  Conditions guaranteeing the existence of a neat embedding of a (completion of a) wMDS into a  complete toric variety are studied, showing that, on the one hand, those which are complete and admitting low Picard number are always projective, hence Mori dream spaces in the sense of Hu-Keel. On the other hand, an example of a wMDS that does not admit any neat embedded \emph{sharp} completion (i.e. Picard number preserving) into a  complete toric variety is given, on the contrary of what Hu and Keel exhibited for a MDS.

Moreover, termination of the Mori minimal model program (MMP) for every divisor and a classification of rational contractions for a complete wMDS are studied, obtaining analogous conclusions as for a MDS.

Finally, we give a characterization of wMDS arising from a small $\Q$-factorial modification of a projective weak $\Q$-Fano variety.

\end{abstract}

\section*{Introduction}

A Mori dream space (MDS) is a projective variety for which Mori's minimal model program (MMP) terminates for every divisor. Historically Mori dream spaces were introduced in the pivotal paper \cite{Hu-Keel} by Hu and Keel, as $\Q $-factorial projective varieties admitting a good Mori chambers decomposition of the Neron-Severi group \cite[Def.~1.10]{Hu-Keel}. This turns out to be equivalent to requiring that a MDS is a $\Q$-factorial projective variety $X$ whose class group $\Cl(X)$ is a finitely generated abelian group and whose Cox ring $\Cox(X)$ is a finitely generated algebra \cite[Prop.~2.9]{Hu-Keel}.

Later many authors studied Mori dream spaces mainly emphasizing their combinatorial properties, rather than their birational behaviour. From this point of view, their projective embedding or completeness or sometimes even their $\Q$-factoriality turned out to be unnecessary hypotheses. In fact, normality and the finite gene\-ra\-tion of both $\Cl(X)$ and $\Cox(X)$ are sufficient hypotheses for obtaining the most interesting combinatorial properties, summarized by a good Mori chambers decomposition and giving rise to the main computational aspects of a MDS: e.g. this is what is assumed to be a  MDS by Hausen and Keicher in their Maple package \texttt{MDSpackage}  \cite[Def.~2.3.2]{KeicherTh}, \cite{MDSpackage}. Such a combinatorial nature of a MDS has been extrapolated by Arzhantsev, Derenthal, Hausen and Laface in their recent book \cite{ADHL} with the concept of a \emph{variety arising from a bunched ring} \cite[\S~3.2]{ADHL}, which can be understood as the combinatorial quintessence of a MDS. It makes then sense asking if and why Hu-Keel results on a MDS can be extended beyond the projective setup.

To avoid confusion with the definition of a MDS given by Hu and Keel,  in this paper a $\Q$-factorial algebraic variety with finite generated $\Cl(X)$ and $\Cox(X)$ will be called a \emph{weak} Mori dream space (wMDS): see Definition~\ref{def:wMDS}. The focus is on a possible extension  to a non-projective setup of
Propositions~1.11 and 2.11 in \cite{Hu-Keel}; namely:
\begin{enumerate}
  \item a MDS $X$ admits a \emph{neat} embedding (see Definition~\ref{def:neat}) into a projective toric variety $Z$ \cite[Prop.~2.11]{Hu-Keel}; moreover, there are as many non isomorphic neat projective embeddings of a MDS as pullbacks of Mori chambers of $Z$ contained in the Nef cone $\Nef(X)$;
  \item Mori's MMP can be carried out for any divisor on a MDS \cite[Prop.~1.11~(1)]{Hu-Keel};
  \item a MDS admits a finite number of \emph{rational contractions} (see Definition~\ref{def:rat-contr}) associated with a fan structure on its Mov cone \cite[Prop.~1.11~(2),(3)]{Hu-Keel}.
\end{enumerate}
Extending (2) and (3) to a wMDS is an application of \cite[Thm.~4.3.3.1]{ADHL}, here recalled by Lemma~\ref{lem:sQm}. Section \S~\ref{sez:birazionale} is entirely devoted to give a proof of (2) and (3), namely in \S~\ref{ssez:MMP} and \S~\ref{ssez:contrraz}, respectively.

The attempt to extend (1) to the case of a wMDS is more interesting. Since in general a wMDS is not a (quasi)-projective algebraic variety, (1) should be understood in terms of a neat embedding of (a suitable closure of) a wMDS into a complete toric variety: this fact does not seem to hold in general! More precisely, section \S~\ref{sez:wMDS} is devoted to developing the following steps:
\begin{itemize}
\item[(a)] describing a canonical neat embedding of a wMDS $X$ into a $\Q$-factorial toric variety $W$; in general, the latter, if exixting, is not even complete;
\item[(b)] finding conditions guaranteeing the existence of a \emph{sharp} completion $Z$ of $W$, where sharp means \emph{Picard number preserving};
\item[(c)] perform a closure of $X$ inside $Z$, when the latter exists, so obtaining a  neat embedding of such a closure into  $Z$; if
    $X$ is a complete wMDS, this gives a neat embedding of $X$ into the  complete toric variety $Z$.
\end{itemize}
Step (a) is known, after the canonical toric embedding of a variety arising from a bunched ring performed in \cite[\S~3.2.5]{ADHL}. Anyway, in the present paper, an alternative proof of this result, for a wMDS endowed with a \emph{Cox basis} (see Definition~\ref{def:Coxgen}) is presented in Theorem~\ref{thm:canonic-emb}, essentially for two reasons: giving an explicit construction of the canonical ambient toric variety $W$, useful for the sequel, and characterizing a wMDS and its canonical toric embedding more algebraically, by means of a suitable presentation of the Cox ring together with the so-called \emph{irrelevant ideal} $\irr(X)$ of $X$, rather than combinatorially by means of properties of the associated bunch of cones.

Step (b) represents the core of the present paper. By Nagata's compactification theorem  e\-ve\-ry algebraic variety can be embedded in a complete algebraic variety \cite[Thm.\,4.3]{Nagata}. Sumihiro proved an equivariant version of this theorem for normal algebraic varieties endowed with an algebraic group action \cite{Sumihiro1},\cite{Sumihiro2}, as toric varieties are. In particular, for toric varieties, it corresponds to a combinatorial completion procedure for fans, as explained by Ewald and Ishida \cite[Thm.~III.2.8]{Ewald96}, \cite{ES}. See also the more recent \cite{Rohrer11} where Rohrer gives a simplified approach to the Ewald-Ishida completion procedure.
In general, all these procedures require the adjunction of some new ray into the fan under completion. This is necessary in dimension $\geq 4$: there are examples of 4-dimensional fans which cannot be completed without the introduction of new rays. Consider the Remark ending up \S~III.3 in \cite{Ewald96} and references therein, for a discussion of this topic; for explicit examples consider \cite[Ex.~2.16]{RT-Pic} and the canonical ambient toric variety here presented in Example~\ref{ex:noncompletabile}.
Adding some new rays necessarily imposes an increasing of the Picard number: we  call \emph{sharp} a completion which does not increase the Picard number. What has been just observed is that, in general, a sharp completion of a  toric variety does not exist. In \cite[Prop.~2.11]{Hu-Keel} Hu and Keel, as already recalled in the previous item (1), showed that the canonical ambient toric variety $W$ of a MDS $X$ admits sharp completions, which are even projective, one for each Mori chamber whose pullback is contained in $\Nef(X)$. This result does no more hold for the ambient toric variety of a wMDS, as Example~\ref{ex:noncompletabile} shows. The key condition should be imposed, in order to guarantee the existence of a sharp completion in the non-projective set up, is the existence of particular cells inside the $\Nef$ cone, called  \emph{filling cells} (see Definition~\ref{def:filling}). This is the content of Theorem~\ref{thm:complete-emb}. Since every Mori chamber, in the sense of Hu-Keel, is a filling cell, such a condition is automatically satisfied when considering a MDS. As a byproduct, directly following from an analogous result for $\Q$-factorial complete toric varieties, jointly obtained with Lea Terracini \cite[Thm.~3.2]{RT-r2proj} and here recalled by Theorem~\ref{thm:r2}, one gets that the Hu-Keel result (1) holds for a complete wMDS with Picard number $r\leq 2$. In fact, every wMDS of this kind is actually a MDS: this is proved in Theorem~\ref{thm:r<3fillable}. The bound on the Picard number can be extended to 3 in the smooth case, as a consequence of Kleinschmidt and Sturmfels results \cite{Kleinschmidt-Sturmfels}.

Step (c) follows from (b), just by taking the closure of $X$ inside a sharp completion $Z$ of the toric ambient variety $W$.

This paper ends up by exhibiting a class of examples of weak Mori dream spaces, namely those varieties which are birational and isomorphic in codimension 1 (i.e. \emph{small $\Q$-factorial modifications}) to a projective weak $\Q$-Fano variety (see Prop.~\ref{prop:wFsQm}). Corollary~\ref{cor:wMDS} and Proposition~\ref{prop:wFsQm} give an extension, beyond the projective setup, of characterizations given by McKernan for smooth Mori dream spaces arising from log Fano
varieties \cite[Lemma~4.10]{McK} and , more recently, by Gongyo, Okawa, Sannai and Takagi for Mori dream spaces arising from Fano type varieties \cite[Thm.~1.1]{GOST}.

\halfline The present paper is organized as follows.

Section \ref{sez:preliminari} is entirely devoted to recall the most part of necessary preliminaries and introduce main notation.

Section \ref{sez:wMDS} is dedicated to study weak Mori dream spaces and their embeddings.
Section~\ref{sez:birazionale} is devoted to discussing the termination of a Mori MMP for every divisor and to classifying rational contractions of a wMDS.

Finally, Section \ref{sez:weakFano} is dedicated to characterizing small $\Q$-factorial modifications of projective weak $\Q$-Fano varieties as those weak Mori dream spaces whose total coordinate space has at most Klt singularities and admitting big and movable anti-canonical class.

\section{Preliminaries and Notation}\label{sez:preliminari}

Throughout the present paper, our ground field will be an algebraic closed field $\K=\overline{\K}$, with $\Char\K=0$.

\subsection{Toric varieties}

Throughout the present paper we will adopt the following definition of a toric variety:
\begin{definition}[Toric variety]\label{def:TV}
  A \emph{toric variety} is a tern $(X,\T,x_0)$ such that:
\begin{itemize}
  \item[(i)] $X$ is an irreducible, normal, $n$-dimensional algebraic variety over $\K$,
  \item[(ii)] $\T\cong(\K^*)^n$ is a $n$-torus freely acting on $X$,
  \item[(iii)] $x_0\in X$ is a special point called the \emph{base point}, such that the orbit map $t\in\T\mapsto t\cdot x_0\in\T\cdot x_0\subseteq X$ is an open embedding.
\end{itemize}
\end{definition}
For standard notation on toric varieties and their defining \emph{fans} we refer to the extensive treatment \cite{CLS}.

\begin{definition}[Morphism of toric varieties]\label{def:TVmorphism} Let $Y$ and $X$ be toric varieties with acting tori $\T_Y$ and $\T_X$ and base points $y_0$ and $x_0$, respectively. A morphism of algebraic varieties $\phi:Y\to X$ is called a \emph{morphism of toric varieties} if
  \begin{itemize}
    \item[(i)] $\phi(y_0)=x_0$\,,
    \item[(ii)] $\phi$ restricts to give a homomorphism of tori $\phi_\T:\T_Y\to\T_X$ by setting $$\phi_\T(t)\cdot x_0=\phi(t\cdot y_0)$$
  \end{itemize}
\end{definition}

The previous conditions (i) and (ii) are equivalent to require that $\phi$ induces a morphism between underling fans, as defined e.g in \cite[\S~3.3]{CLS}.

\subsubsection{List of notation}\label{sssez:lista}
\begin{eqnarray*}
  &M,N,M_{\R},N_{\R}& \text{denote the \emph{group of characters} of $\T$, its dual group}\\
  && \text{and their tensor products with $\R$, respectively;} \\
  &\Si\subseteq \mathfrak{P}(N_{\R})& \text{is the fan defining $X$;}\\
  &&\text{$\mathfrak{P}(N_{\R})$ denotes the power set of $N_{\R}$} \\
  &\Si(i)& \text{is the \emph{$i$--skeleton of $\Si$};}\\
  &\langle\v_1,\ldots,\v_s\rangle\subseteq\N_{\R}& \text{cone generated by $\v_1,\ldots,\v_s\in N_{\R}$;}\\
  && \text{if $s=1$ this cone is called the \emph{ray} generated by $\v_1$;} \\
  &\mathcal{L}(\v_1,\ldots,\v_s)\subseteq N& \text{sublattice spanned by $\v_1,\ldots,\v_s\in N$\,;}\\
\end{eqnarray*}
Let $A\in\mathbf{M}(d,m;\Z)$ be a $d\times m$ integer matrix, then
\begin{eqnarray*}
  &\mathcal{L}_r(A)\subseteq\Z^m& \text{is the sublattice spanned by the rows of $A$;} \\
  &\mathcal{L}_c(A)\subseteq\Z^d& \text{is the sublattice spanned by the columns of $A$;} \\
  &A_I\,,\,A^I& \text{$\forall\,I\subseteq\{1,\ldots,m\}$ the former is the submatrix of $A$ given by}\\
  && \text{the columns indexed by $I$ and the latter is the submatrix}\\
  && \text{of $A$ whose columns are indexed by the complementary }\\
  && \text{subset $\{1,\ldots,m\}\backslash I$;} \\
  &\text{\emph{positive}}& \text{a matrix (vector) whose entries are non-negative.}
\end{eqnarray*}
Given an integer matrix $V=(\v_1,\ldots,\v_{m})\in\mathbf{M}(n,m;\Z)$, then
\begin{eqnarray*}
  &\langle V\rangle=\langle\v_1,\ldots,\v_{m}\rangle\subseteq N_{\R}& \text{cone generated by the columns of $V$;} \\
  &\SF(V)=\SF(\v_1,\ldots,\v_{m})& \text{set of all rational simplicial fans $\Si$ such that}\\
  && \text{$\Sigma(1)=\{\langle\v_1\rangle,\ldots,\langle\v_{m}\rangle\}\subseteq N_{\R}$ and} \\ && \text{$|\Si|=\langle V\rangle$ \cite[Def.~1.3]{RT-LA&GD}.}\\
  & \I_\Si& :=\{I\subseteq\{1,\dots,m\}\,|\,\langle V_I\rangle\in\Si\}\\
  &\G(V)& \text{is a \emph{Gale dual} matrix of $V$ \cite[\S~3.1]{RT-LA&GD};} \\
\end{eqnarray*}

\subsection{$F$ and $W$-matrices}

\begin{definition}[$F$-matrix, Def.~3.10 in \cite{RT-LA&GD}]\label{def:Fmatrice} An \emph{$F$--matrix} is a $n\times m$ matrix  $V$ with integer entries, satisfying the conditions:
\begin{itemize}
\item[(a)] $\rk(V)=n$;
\item[(b)] $V$ is \emph{$F$--complete} i.e. $\langle V\rangle=N_{\R}\cong\R^n$ \cite[Def.~3.4]{RT-LA&GD};
\item[(c)] all the columns of $V$ are non zero;
\item[(d)] if ${\bf  v}$ is a column of $V$, then $V$ does not contain another column of the form $\lambda  {\bf  v}$ where $\lambda>0$ is real number.
\end{itemize}
An $F$--matrix $V$ is called \emph{reduced} if every column of $V$ is composed by coprime entries \cite[Def.~3.13]{RT-LA&GD}.
\end{definition}
The most significant example of a $F$-matrix is given by a matrix $V$ whose columns are    integral vectors generating the rays of the $1$-skeleton $\Sigma(1)$ of a rational complete fan $\Sigma$. In the following $V$ will be called a \emph{fan matrix} of $\Sigma$; when every column of $V$ is composed by coprime entries, it will be called a \emph{reduced fan matrix}. For a detailed definition see
\cite[Def.~1.3]{RT-LA&GD}

\begin{definition}[$W$-matrix, Def.~3.9 in \cite{RT-LA&GD}]\label{def:Wmatrice} A \emph{$W$--matrix} is an $r\times m$ matrix $Q$  with integer entries, satisfying the following conditions:
\begin{itemize}
\item[(a)] $\rk(Q)=r$;
\item[(b)] ${\mathcal L}_r(Q)$ does not have cotorsion in $\Z^{m}$;
\item[(c)] $Q$ is \emph{$W$--positive}, that is, $\mathcal{L}_r(Q)$ admits a basis consisting of positive vectors \cite[Def.~3.4]{RT-LA&GD}.
\item[(d)] Every column of $Q$ is non-zero.
\item[(e)] ${\mathcal L}_r(Q)$   does not contain vectors of the form $(0,\ldots,0,1,0,\ldots,0)$.
\item[(f)]  ${\mathcal L}_r(Q)$ does not contain vectors of the form $(0,a,0,\ldots,0,b,0,\ldots,0)$, with $ab<0$.
\end{itemize}
A $W$--matrix is called \emph{reduced} if $V=\G(Q)$ is a reduced $F$--matrix \cite[Def.~3.14, Thm.~3.15]{RT-LA&GD}
\end{definition}
The most significant example of a $W$-matrix $Q$ is given by a Gale dual matrix of a fan matrix $V$, that is $Q=\G(V)$. In this case $Q$ will also be called a \emph{weight matrix}. If $Q$ is also a reduced $W$-matrix then it is a \emph{reduced weight matrix}.

\subsection{Cox sheaf and algebra of an algebraic variety}\label{ssez:Cox}

For what concerning the present topic we will essentially adopt notation introduced in the extensive book \cite{ADHL}, to which the interested reader is referred for any further detail.

Let $X$ be an irreducible and normal, algebraic variety of dimension $n$ over $\K$. The group of Weil divisors on $X$ is denoted by $\Div(X)$\,: it is the free group generated by prime divisors of $X$. For $D_1,D_2\in\Div(X)$, $D_1\sim D_2$ means that they are linearly equivalent. The subgroup of Weil divisors linearly equivalent to 0 is denoted by $\Div_0(X)\leq\Div(X)$. The quotient group $\Cl(X):=\Div(X)/\Div_0(X)$ is called the \emph{class group}, giving the following short exact sequence of $\Z$-modules
\begin{equation}\label{Wdivisori}
  \xymatrix{0\ar[r]&\Div_0(X)\ar[r]&\Div(X)\ar[r]^-{d_X}&\Cl(X)\ar[r]&0}
\end{equation}
Given a divisor $D\in\Div(X)$, its class $d_X(D)$ is often denoted by $[D]$, when no confusion may arise.

\subsubsection{Assumption}\label{ipotesi} In the following, $\Cl(X)$ is assumed to be a \emph{finitely generated} (f.g.) abelian group of rank $r:=\rk(\Cl(X))$. Then $r$ is called either the \emph{Picard number} or the \emph{rank} of $X$.
Moreover we will assume that every invertible global function is constant i.e.
\begin{equation}\label{costanti}
  H^0(X,\cO_X^*)\cong\K^*\,.
\end{equation}
The latter condition is clearly satisfied when $X$ is complete.

\subsubsection{Choice}\label{ssez:K} Choose a f.g. subgroup $K\leq\Div(X)$ such that
$$\xymatrix{d_K:=d_X|_K:K\ar@{>>}[r]&\Cl(X)}$$
is an \emph{epimorphism}. Then $K$ is a free group of rank $m\geq r$ and (\ref{Wdivisori}) induces the following exact sequence of $\Z$-modules
\begin{equation}\label{Kdivisori}
  \xymatrix{0\ar[r]&K_0\ar[r]&K\ar[r]^-{d_K}&\Cl(X)\ar[r]&0}
\end{equation}
where $K_0:=\Div_0(X)\cap K=\ker(d_K)$.

\begin{definition}[Sheaf of divisorial algebras, Def.~1.3.1.1 in \cite{ADHL}] The \emph{sheaf of divisorial algebras} associated with the subgroup $K\leq\Div(X)$ is the sheaf of $K$-graded $\cO_X$-algebras
\begin{equation*}
  \cS:=\bigoplus_{D\in K}\cS_D\,,\quad \cS_D:=\cO_X(D)\,,
\end{equation*}
where the multiplication in $\cS$ is defined by multiplying homogeneous sections in the field of functions $\K(X)$.
\end{definition}

\subsubsection{Choice}\label{ssez:chi} Choose a character $\chi:K_0\to\K(X)^*$ such that
\begin{equation*}
  \forall\,D\in K_0\quad D=(\chi(D))
\end{equation*}
where $(f)$ denotes the principal divisor defined by the rational function $f\in\K(X)^*$. Consider the ideal sheaf $\I_\chi$ locally defined by sections $1-\chi(D)$ i.e.
\begin{equation*}
  \Ga(U,\I_\chi)=\left((1-\chi(D))|_U\,|\,D\in K_0\right)\subseteq\Ga(U,\cS)\,.
\end{equation*}
This induces the following short exact sequence of $\cO_X$-modules
\begin{equation}\label{Sdivisori}
  \xymatrix{0\ar[r]&\I_\chi\ar[r]&\cS\ar[r]^-{\pi_\chi}&\cS/\I_\chi\ar[r]&0}
\end{equation}

\begin{definition}[Cox sheaf and Cox algebra, Construction~1.4.2.1 in \cite{ADHL}]\label{def:CoxRings}
Keeping in mind the exact sequence (\ref{Sdivisori}), the \emph{Cox sheaf} of $X$, associated with $K$ and $\chi$, is the quotient sheaf $\mathcal{C}ox:=\cS/\I_\chi$ with the $\Cl(X)$-grading
\begin{equation}\label{graduazione}
  \cox:=\bigoplus_{\d\in \Cl(X)}\cox_{\d}\,,\quad \cox_{\d}:=\pi_\chi(\cS_\d)\,\quad\cS_\d:=\left(\bigoplus_{D\in d_K^{-1}(\d)}\cS_D\right)\,.
\end{equation}
Passing to global sections, one gets the following \emph{Cox algebra} (usually called Cox \emph{ring}) of $X$, associated with $K$ and $\chi$,
\begin{equation*}
  \Cox(X):=\cox(X)= \bigoplus_{\d\in \Cl(X)}\Ga(X,\cox_{\d})\,.
\end{equation*}
\end{definition}

\begin{remark}[\cite{ADHL}, Lemma 1.4.3.1]\label{rem:Ichi-omogeneo}
  With respect to the $\Cl(X)$-graded decompositions given in (\ref{graduazione}), the exact sequence (\ref{Sdivisori}) behaves coherently, that is $\I_\chi$ is a $\Cl(X)$-homogeneous sheaf of ideals in $\cS$ admitting a well-defined $\Cl(X)$-graded decomposition $\I_\chi=\bigoplus_{\d\in\Cl(X)}\I_{\chi,\d}$ such that $\cox_\d\cong\cS_\d/\I_{\chi,\d}$.
\end{remark}

\begin{proposition}\label{prop:Ichi}
  For every $D,D'\in K$, $D\sim D'$ if and only if there exists an isomorphism $\psi:H^0(X,\cO_X(D))\stackrel{\cong}{\longrightarrow}H^0(X,\cO_X(D'))$ such that
  \begin{equation*}
    \forall\,f\in H^0(X,\cO_X(D))\quad f-\psi(f)\in H^0(X,\I_\chi)\,.
  \end{equation*}
\end{proposition}

\begin{proof} Assume $D\sim D'$.
  Then $E:=D-D'\in K_0$, meaning that $E=(\chi(E))$. Then define
  \begin{equation*}
    \forall\,f\in H^0(X,\cO_X(D))\quad \psi(f):=f\cdot\chi(E)\,.
  \end{equation*}
  $\psi$ is well defined:
  $$(\psi(f))+D'=(f)+E+D'=(f)+D\geq 0\ \Longrightarrow\ \psi(f)\in H^0(X,\cO_X(D'))\,.$$
  $\psi^{-1}$ is well defined by setting
  \begin{equation*}
    \forall\,g\in H^0(X,\cO_X(D'))\quad \psi^{-1}(g):=g\cdot\chi(-E)\in H^0(X,\cO_X(D))\,.
  \end{equation*}
  Then $\psi$ gives an isomorphism and
  \begin{equation*}
    \forall\,f\in H^0(X,\cO_X(D))\quad f-\psi(f)=f\cdot(1-\chi(E))\in H^0(X,\I_\chi)\,.
  \end{equation*}
  Viceversa, the necessary condition follows from Remark~\ref{rem:Ichi-omogeneo}. In fact $f-\psi(f)\in H^0(X,\I_\chi)$ means that $[D]=[D']$, by the $\Cl(X)$-graded exact sequence (\ref{Sdivisori}).
\end{proof}

\begin{remarks}\label{rems}\hfill\rm{
\begin{enumerate}
  \item \cite[Prop.~1.4.2.2]{ADHL} Depending on choices \ref{ssez:K} and \ref{ssez:chi}, both Cox sheaf and algebra are not canonically defined. Anyway, given two choices $K,\chi$ and $K',\chi'$ there is a graded isomorphism of $\cO_X$-modules $$\cox(K,\chi)\cong\cox(K',\chi')\,.$$
  \item For any open subset $U\subseteq X$, there is a canonical isomorphism
  $$\xymatrix{\Ga(U,\cS)/\Ga(U,\I_\chi)\ar[r]^-{\cong}&\Ga(U,\cox)}\,.$$
  In particular $\Cox(X)\cong H^0(X,\cS)/H^0(X,\I_\chi)$. This fact, jointly with Proposition~\ref{prop:Ichi}, explains why $\I_\chi$ is the right sheaf of ideals giving a precise meaning to the usual ambiguous writing
  $$\Cox(X)\cong \bigoplus_{[D]\in\Cl(X)}H^0(X,\cO_X(D))\,.$$
  \item Choice \ref{ssez:chi} is needed to fixing a unique section $f_D\in\Ga(X,\cS_{-D})$, for each principal divisor $D\in K_0$. Both choices \ref{ssez:K} and \ref{ssez:chi} are unnecessary if a special point $x\in X$ is chosen. In this case both $K$ and $\chi$ can be canonically assigned by setting
      \begin{eqnarray*}
        K&:=&K^x = \{D\in\Div(X)\,|\, x\not\in \Supp(D)\} \\
        \chi&:=&\chi^x:K^x_0 \longrightarrow \K(X)^*\,,\quad D=(\chi^x(D))\,,\,\chi^x(D)(x)=1\,.
      \end{eqnarray*}
      Then $\cox(X,x)$ is canonical \cite[Construction~1.4.2.3]{ADHL}. This is the case e.g. when $X$ is a normal toric variety whose fan is non-degenerate (i.e. $H^0(X,\cO^*_X)\cong\K^*$, as assumed in Assumption~\ref{ipotesi}, display (\ref{costanti})) and $x_0$ is the \emph{base point} fixing the open embedding $\T\to\T\cdot x_0\subseteq X$ of the acting torus $\T=(\K^*)^n$ \cite[Def.~2.1.1.1]{ADHL}.
  \item In case $X$ is a normal toric variety with non-degenerate fan, the canonical choices of $K^{x_0},\chi^{x_0}$ can be furtherly specialized as follows. Let $\Si$ be a fan of $X$ and $x_{\rho}$ the \emph{distinguished point} of a ray $\rho\in\Si(1)$ (for a definition see e.g. \cite[\S~3.2]{CLS}). Let $D_{\rho}$ be the associated torus invariant divisor i.e. $D_{\rho}=\overline{\T\cdot x_{\rho}}\subseteq X$.
      Then
      \begin{equation*}
        \bigcup_{\rho\in\Si(1)}D_{\rho}=X\backslash (\T\cdot x_0)\ \Rightarrow\  K^{x_0}\geq\bigoplus_{\rho\in\Si(1)}\Z\cdot D_{\rho}=:\Div_{\T}(X)\leq \Div(X)
      \end{equation*}
      $\Div_{\T}(X)$ is called the subgroup of \emph{torus invariant Weil divisors} of $X$. It is a well known fact that every class in $\Cl(X)$ admits a tours invariant representative (see e.g. \cite[Thm.~4.1.3]{CLS}), giving a surjection $$\xymatrix{d_X|_{\Div_{\T}(X)}:\Div_{\T}(X)\ar@{->>}[r]&\Cl(X)}\,.$$
      Then we get the canonical choice $K=\Div_{\T}(X)$. In this case $K_0=\ker(d_X)\cap\Div_{\T}$ is the subgroup of torus invariant principal divisors, which are principal divisors $D$ admitting, as a defining function, a Laurent monomial associated with an exponent $\m\in M=\Hom(\T,\Z)$ and well defined up to a factor $k\in H^0(X,\cO_X^*)\cong\K^*$. Namely, in Cox coordinates
      \begin{equation*}
      D=\left(k\prod_{j=1}^{|\Si(1)|}x_j^{<\m,\v_j>}\right)
      \end{equation*}
 where $\v_j$ is a primitive generator of the ray $\langle\v_j\rangle\in \Si(1)$ i.e. a column of a fan matrix $V$ of $X$. Normalizing $k$ on the base point $x_0$ of $X$, one gets the natural choice:
\begin{equation*}
        \xymatrix{\chi:=\chi^{x_0}|_{K_0}: D\in K_0\ar@{|->}[r]&\x^\m:=\prod_{j=1}^{|\Si(1)|}x_j^{<\m,\v_j>}\in\K(X)^*}
\end{equation*}
Then $\cS=\bigoplus_{D\in \Div_{\T}(X)}\cO_X(D)$ and $\I_\chi(U)=(\{(1-\x^\m)|_U\,|\,\m\in M\})$, for every open subset $U\subseteq X$.\\
      Moreover the exact sequence (\ref{Kdivisori}) gives rise, under our hypothesis, to the following well known exact sequence of abelian groups
      \begin{equation}\label{Tdivisori}
  \xymatrix{0\ar[r]&M\ar[r]^-{div_X}_-{V^T}&K=\Div_\T(X)\ar[r]^-{d_K}_-{Q\oplus\Ga}&\Cl(X)\ar[r]&0}
\end{equation}
      where $V$ and $Q=\G(V)$ are a fan matrix and a weight matrix, respectively, of $X$, as defined after Definitions \ref{def:Fmatrice} and \ref{def:Wmatrice}:
      \begin{itemize}
        \item  the  transposed $V^T$ represents the morphism $div_X$ assigning to a character $\m\in M$ the associated principal torus invariant divisor $div_X(\m)=(\x^\m)$, with respect to a suitable basis of $M$ and the basis of torus invariant divisors $\{D_\rho\}_{\rho\in \Si(1)}$; if $X$ is complete then $V$ is an $F$-matrix;
        \item $Q$ is a Gale dual matrix to $V$, and representing, together with a \emph{torsion matrix} $\Ga$, the morphism $d_K$ assigning a linear equivalence class $d_K(D)=[D]$ to every torus invariant divisor $D\in \Div_\T(X)$, with respect to the basis $\{D_\rho\}$ and a suitable choice of generators of $\Cl(X)$: namely, given a decomposition $\Cl(X)=F\oplus\Tors(X)$, where $\Tors(X)$ is the canonical torsion subgroup of $\Cl(X)$ and $F$ a free part, also the class morphism splits as
            \begin{equation*}
              d_K=f_K\oplus\tau_K\quad\text{with}\quad d_K:\Div_\T(X)\to F\ ,\ \tau_K:\Div_\T(X)\to\Tors(X)
            \end{equation*}
            and $Q$ and $\Ga$ are their representative matrices; if $X$ is complete then $Q$ is a $W$-matrix.
      \end{itemize}
      Therefore, the kernel $K_0$ admits the following interpretation
      \begin{equation*}
        K_0=\im(div_X)=\{V^T\m\,|\,\m\in M\}\,\ \text{and}\ \xymatrix{\chi:V^T\m \in K_0\ar@{|->}[r]&\x^{\m}\in\K(X)^*}\,.
      \end{equation*}
\end{enumerate}}
\end{remarks}

\begin{proposition}\label{prop:Coxtorico}
  The Cox algebra of a toric variety without torus factor $X(\Si)$ is a polynomial algebra formally generated by the rays of its 1-scheleton $\Si(1)$, that is
\begin{eqnarray*}
  \Cox(X)&\cong&\frac{\bigoplus_{D\in\Div_\T(X)}H^0(X,\cO_X(D))}{H^0(X,\I_\chi)}\\
  &\cong&\bigoplus_{\d\in\Cl(X)}\left(\bigoplus_{D\in\d\,,\,D\geq 0}\K\cdot\x^D\right)\ \cong\ \K\left[\{x_\rho\}_{\rho\in\Si(1)}\right]
\end{eqnarray*}
\end{proposition}

\begin{remark}
  The previous Proposition \ref{prop:Coxtorico} was proved earlier by Batyrev and Mel'nikov \cite{Batyrev-Melnikov} and it is clearly implicit in the Cox's quotient construction given in \cite{Cox}. Proposition~\ref{prop:Coxtorico} admits also a converse statement, so giving an ``if and only if'' characterization when:
  \begin{itemize}
    \item $X$ is smooth with $\Cl(X)$ f.g. \cite[Thm.~3.1]{Jaczewski}, \cite[Cor.~2.10]{Hu-Keel},
    \item $X$ is normal and projective with a lattice contained in $\Pic(X)$ \cite[Thm.~1.5]{Kedzierski-Wisniewski}
    \item see also the recent \cite[Thm.~1.2]{B&McK&S&Z} for a log relative version with $(X,\Delta)$ a toric log pair admitting log canonical singularities such that $-(K_X+\De)$ is nef and $\De$ admits a \emph{decomposition of complexity less than 1} \cite[Def.~1.1]{B&McK&S&Z}.
  \end{itemize}
\end{remark}

\section{Weak Mori dream spaces (wMDS) and their embedding}\label{sez:wMDS}

In the literature Mori dream spaces (MDS) come with a required projective embedding essentially for their optimal behavior with respect to the termination of Mori program. Actually this assumption is not necessary to obtain main properties of MDS, like e.g. their toric embedding, chamber decomposition of their moving and pseudo-effective cones and even termination of Mori program, for what this fact could mean for a  non-projective algebraic variety.

\begin{definition}[wMDS]\label{def:wMDS} An irreducible and \emph{$\Q$-factorial} algebraic va\-rie\-ty $X$ (i.e. normal and such that a suitable integer multiple of a Weil divisor is a Cartier divisor) sa\-ti\-sfying assumption~\ref{ipotesi} is called a \emph{weak Mori dream space} (wMDS) if $\Cox(X)$ is a finitely generated $\K$-algebra. A projective (hence complete) wMDS is called a \emph{Mori dream space} (MDS).
\end{definition}

\subsection{Total coordinate and characteristic spaces}\label{ssez:quotient}

Let us open the present section with the following

\begin{remarks}\label{rems:wMDS}\hfill\rm{
\begin{enumerate}
  \item The Cox sheaf $\cox=\cS/\I_\chi$ of a wMDS $X$ is \emph{locally of finite type}, that is there exists a finite affine covering $\bigcup_iU_i=X$ such that $\cox(U_i)$ are finitely generated $\K$-algebras \cite[Propositions~1.6.1.2, 1.6.1.4]{ADHL}.
  \item As a consequence of the previous part (1), the \emph{relative spectrum} of $\cox$ \cite[Ex.~II.5.17]{Hartshorne},
      \begin{equation}\label{relspectra}
        \widehat{X}=\Spec_X(\cox)\stackrel{p_X}{\longrightarrow} X
      \end{equation}
        is an irreducible normal and quasi-affine variety, coming with an actions of the quasi-torus $G:=\Hom(\Cl(X),\K^*)$, whose quotient map is realized  by the canonical morphism in (\ref{relspectra}) \cite[\S~1.3.2 , Construction~1.6.1.5]{ADHL}.
   \item Consider $\overline{X}:=\Spec(\Cox(X))$ which is an irreducible and normal, affine variety. Then there exists an open embedding
       \begin{equation*}
          j_X:\widehat{X}\hookrightarrow\overline{X}
       \end{equation*}
       The action of the quasi-torus $G$ extends to $\overline{X}$ in such a way that $j_X$ turns out to be equivariant.
   \item Since $\Cox(X)$ is a finitely generated $\K$-algebra, up to the choice of a set of generators $\frak{X}=(x_1,\ldots,x_m)$, we get
       \begin{equation*}
         \Cox(X)\cong\K[\X]/I
       \end{equation*}
       being $I\subseteq\K[\frak{X}]:=\K[x_1,\ldots,x_m]$ a suitable ideal of relations. Calling $\overline{W}:=\Spec\K[\X]\cong\K^m$, the canonical surjection
       \begin{equation}\label{can-surj}
         \xymatrix{\pi_\X:\K[\X]\ar@{->>}[r]&\Cox(X)}
       \end{equation}
       gives rise to a closed embedding $\overline{i}:\overline{X}\hookrightarrow\overline{W}\cong\K^m$, depending on the choice of $(K,\chi,\X)$.
\end{enumerate}}
\end{remarks}

In the following definitions we will consider a quasi-torus $G$ acting on an irreducible and normal algebraic variety $Y$.

\begin{definition}[Good and geometric quotients]\label{def:quozienti} A surjective morphism $p:Y\twoheadrightarrow X$ is called a \emph{good quotient} for the $G$-action on $Y$ if\begin{itemize}
    \item[(i)] $p$ is \emph{affine}, that is $p^{-1}(U)\subseteq Y$ is affine for every open affine subset $U\subseteq X$,
    \item[(ii)] $p$ is \emph{$G$-invariant}, that is $p$ is constant along every $G$-orbit,
    \item[(iii)] the pull back $p^*:\cO_X\stackrel{\cong}{\longrightarrow} p_*\cO_Y^G$ is an isomorphism.
  \end{itemize}
  A good quotient $p:Y\twoheadrightarrow X$ is called a \emph{geometric quotient} if its fibers are \emph{precisely} the $G$-orbits.
\end{definition}
For ease, in the definition above, we required $p$ to be surjective: actually this is an overabundant hypothesis \cite[Cor.~1.2.3.7]{ADHL}.

\begin{definition}[$1$-free action]\label{def:1-free} A $G$-action on $Y$ is called \emph{free in codimension 1}, or simply \emph{$1$-free}, if it induces a good quotient $p:Y\twoheadrightarrow X$ and there exists an open subset $V\subseteq Y$ such that
\begin{itemize}
  \item[(i)] the complement $Y\backslash V$ has codimension greater than or equal to 2,
  \item[(ii)] $G$ acts freely on $V$,
  \item[(iii)] for every $x$ in $V$ the orbit $G\cdot x$ is closed in $Y$.
\end{itemize}
\end{definition}

\begin{definition}[Stable, semi-stable and unstable loci] Consider a $G$-action on $Y$. The subset $Y^{ss}\subseteq Y$ of \emph{semi-stable points}, with respect to the given $G$-action, is the greatest subset of $Y$ such that $p|_{Y^{ss}}:Y^{ss}\twoheadrightarrow p(Y^{ss})$ is a good quotient. The subset $Y^s\subseteq Y^{ss}$ of \emph{stable points} is the greatest subset of $Y$ such that $p|_{Y^{s}}:Y^{s}\twoheadrightarrow p(Y^{s})$ is a geometric quotient. The subset of \emph{unstable points} is the complement $Y^{nss}:=Y\setminus Y^{ss}$.
\end{definition}

\begin{theorem}[Cox Theorem for a wMDS]\label{thm:K-emb} Let $X$ be a wMDS and consider the quasi-torus action of $G=\Hom(\Cl(X),\K^*)$ on the affine va\-rie\-ty $\overline{X}=\Spec(\Cox(X))$. Then $j_X(\widehat{X})=\overline{X}^s=\overline{X}^{ss}$, giving rise to a $1$-free, geometric quotient $p_X:\widehat{X}\twoheadrightarrow X$ such that $$(p_X)_*(\cO_{\widehat{X}})\cong\cox\quad,\quad(p_X)^*:\cO_X\stackrel{\cong}{\longrightarrow} \cox^G:=(p_X)_*\cO_{\widehat{X}}^G\,.$$
\end{theorem}

\begin{proof} $p_X$ gives a good quotient, with $(p_X)_*(\cO_{\widehat{X}})\cong\cox$, by the relative spectrum  construction of $\widehat{X}$, as explained in Remarks~\ref{rems:wMDS}~(2),~(3). In particular this gives $\overline{X}^{ss}=j_X(\widehat{X})$ and condition (iii) in Definition~\ref{def:quozienti} implies that
$$\xymatrix{(p_X)^*:\cO_X\ar[r]^-\cong & (p_X)_*(\cO_{\widehat{X}}^G)=:\cox^G}\,.$$
Since $X$ is a wMDS, it is $\Q$-factorial, hence $p_X$ turns out to give a geometric quotient by \cite[Cor.~1.6.2.7(ii)]{ADHL}. Then $\overline{X}^s=j_X(\widehat{X})$. Finally \cite[Prop.~1.6.1.6(i)]{ADHL} proves that $G$ acts freely on $p_X^{-1}(X_{reg})$ where $X_{reg}\subseteq X$ is the open subset of smooth points. Since $X$ is normal, $\codim_X(X\setminus X_{reg})\geq 2$ and \cite[Prop.~1.6.1.6(ii)]{ADHL} proves that $\codim_{\widehat{X}}(\widehat{X}\setminus p_X^{-1}(X_{reg}))\geq 2$. This is enough to conclude that the $G$-action on $\widehat{X}$ is 1-free as $p_X$ is a geometric quotient.
\end{proof}

\subsubsection{Nomenclature}\label{sssez:nomenclatura}\hfill
\begin{enumerate}
  \item The relative spectrum $\widehat{X}$ introduced in Remark~\ref{rems:wMDS}~(2) is called the \emph{characteristic spaces} of $X$.
  \item The quasi-torus $G$ introduced in Remark~\ref{rems:wMDS}~(2) is called a \emph{characteristic quasi-torus} of the wMDS $X$. Notice that if $\Cl(X)$ is torsion free then $G$ is actually a torus.
  \item The spectrum $\overline{X}$ introduced in Remark~\ref{rems:wMDS}~(3) is called the \emph{total coordinate spaces} of $X$.
\end{enumerate}

\subsection{Irrelevant loci and ideals}\label{ssez:Irr}

Recall Remarks~\ref{rems:wMDS}~(3),~(4), the open embedding $j_X$ of the characteristic space of a wMDS into its total coordinate space and the surjection $\pi_\X$, defined in (\ref{can-surj}), associated with the choice of a finite set of generators of the Cox ring.

\begin{definition}[Irrelevant loci and ideals]\label{def:irr}
  Let $X$ be a wMDS. The \emph{irrelevant locus} of a total coordinate space $\overline{X}$ of $X$ is the Zariski closed subset given by the complement $B_X:=\overline{X}\setminus j_X(\widehat{X})$. Since $\overline{X}$ is affine, the irrelevant locus $B_X$ defines an \emph{irrelevant ideal} of the Cox algebra $\Cox(X)$, as
  \begin{equation*}
    \mathcal{I}rr(X):=\left(f\in\Cox(X)_\d\,|\, \d\in\Cl(X)\ \text{and}\ f|_{B_X}=0 \right)\subseteq \Cox(X)\,.
  \end{equation*}
  Analogously, after the choice of a set $\X$ of generators of $\Cox(X)$, consider the \emph{lifted irrelevant ideal} of $X$
  \begin{equation*}
    \widetilde{\I rr}:=\pi_\X^{-1}(\I rr(X))\subseteq\K[\X]\,.
  \end{equation*}
  The associated zero-locus $\widetilde{B}=\mathcal{V}(\widetilde{\I rr})\subseteq \Spec(\K[\X])=:\overline{W}$ will be called the \emph{lifted irrelevant locus} of $X$.
\end{definition}

\begin{proposition}\label{prop:irr} The following facts hold:
\begin{enumerate}
  \item $\mathcal{I}rr(X)=\left(f\in\Cox(X)_\d\,|\, \d\in\Cl(X)\ \text{and}\ \overline{X}_f:=\overline{X}\setminus\{f=0\}\ \text{is affine} \right)$\,;
  \item $\widetilde{\irr}=\left(P\in\K[\X]\,|\, \overline{W}_P\ \text{is affine} \right)$\,;

  \item the definition of $\widetilde{\irr}$ gives:
      \quad$I\subseteq\widetilde{\irr}\ \Longrightarrow\ \widetilde{B}\subseteq\im(\overline{i})\,;$\\
      then, under the isomorphism $\Cox(X)\cong \K[\X]/I$, it turns out that
$$
\irr(X)\cong \widetilde{\irr}\left/I\right.\quad\text{i.e.}\quad\overline{i}\left(B_X\right)=\widetilde{B}\,.
$$

\end{enumerate}
\end{proposition}

\begin{proof}
  (1) follows immediately from the definition. (2) and (3) are consequences of (1) and the definition of $\widetilde{\irr}$.
\end{proof}

\subsection{The $\X$-canonical toric embedding}\label{ssez:canon-emb}

Let $X$ be a wMDS and $\Cox(X)$ be its Cox ring. Recall that the latter is a graded $\K$-algebra over the class group $\Cl(X)$ of $X$.

\begin{definition}[Cox generators and bases]\label{def:Coxgen} Given a set $\X$ of generators of $\Cox(X)$, as done in Remark~\ref{rems:wMDS}~(4), an element $x\in\X$  is called a \emph{Cox generator} if it is \emph{$\Cl(X)$-prime}, in the sense of \cite[Def.~1.5.3.1]{ADHL}, that is:
\begin{itemize}
  \item $x$ is a non-zero, non-unit element of
$\Cox(X)$ and there exists $\d\in\Cl(X)$ such that $x\in\Cox(X)_\d$ (i.e. $x$ is \emph{homogeneous}) and
  \begin{equation*}
    \forall\,\d_1,\d_2\in\Cl(X)\,,\,\forall\,f_1\in\Cox(X)_{\d_1}\,,\,\forall\,f_2\in\Cox(X)_{\d_2}\quad x\,|\,f_1f_2\ \Longrightarrow\ x\,|\,f_1\ \text{or}\ x\,|\,f_2\,.
  \end{equation*}
\end{itemize}
If $\X$ is entirely composed by Cox generators then it is called a \emph{Cox basis} of $\Cox(X)$ if it has \emph{minimum cardinality}.
\end{definition}

\begin{theorem}[$\X$-canonical toric embedding]\label{thm:canonic-emb}
Let $X$ be a wMDS endowed with a Cox basis $\X$ of $\Cox(X)$. Then there exists a closed embedding $i:X\hookrightarrow W$ into a $\Q$-factorial and non-degenerate toric variety $W$, fitting into the following commutative diagram
\begin{equation}\label{embediag}
    \xymatrix{\overline{X}\ar@/^2pc/[rrr]^-{\overline{i}}&\widehat{X}\ar@{_{(}->}[l]_-{j_X}\ar@{>>}[d]^-{p_X}
    \ar@{^(->}[r]^-{\widehat{i}}
                &\widehat{W}\ar@{^(->}[r]^-{j_W}\ar@{>>}[d]^-{p_W}&\overline{W}\\
                &{X}\ar@{^(->}[r]^-{i}&{W}&}
  \end{equation}
where
\begin{enumerate}
  \item $\overline{W}=\Spec\K[\X]$, as in Remark~\ref{rems:wMDS}~(4),
  \item $\widehat{W}:=\overline{W}\backslash\widetilde{B}$ is a Zariski open subset and $j_W:\widehat{W}\hookrightarrow\overline{W}$ is the associated open embedding,
  \item $\widehat{i}:=\overline{i}|_{\widehat{X}}$\,,
  \item $p_W:\widehat{W}\twoheadrightarrow W$ is a 1-free geometric quotient by an action of the quasi-torus $G=\Hom(\Cl(X),\K^*)$ on the affine va\-rie\-ty $\overline{W}=\Spec(\K[\X])$, \wrt $\widehat{i}$ turns out to be equivariant. Then $j_W(\widehat{W})=\overline{W}^s=\overline{W}^{ss}$ and $$(p_W)^*:\cO_W\stackrel{\cong}{\longrightarrow} (p_W)_*\cO_{\widehat{W}}^G$$
\end{enumerate}
\end{theorem}

\begin{proof}
  Let $\X=\{x_1,\ldots,x_m\}$ be a Cox basis of $\Cox(X)$. Then the surjection $\pi_\X:\K[\X]\twoheadrightarrow \Cox(X)$ gives rise to the closed embedding $\overline{i}:\overline{X}\hookrightarrow\overline{W}\cong\K^m$ explicitly obtained by evaluating the Cox generators i.e.
  $$\xymatrix{x\in \overline{X}\ar@{|->}[r]&(x_1(x),\ldots,x_m(x))\in \K^m}\,.$$
 The fact that $x_i$ is a Cox generator implies that $\overline{x}_i=x_i+I$ is homogeneous, that is there exists a class $\d_i\in \Cl(X)$ such that $\overline{x}_i\in\Cox(X)_{\d_i}$. We can then define an action of the quasi-torus $G=\Hom(\Cl(X),\K^*)$ on $\overline{W}$ by multiplication, as follows,
$$\xymatrix{(g,z)\in G\times\overline{W}\ar@{|->}[r]&(\chi_1(g)z_1,\ldots,\chi_m(g)z_m)\in\overline{W}}$$
where $\chi_i:G\to\K^*$ is the character defined by setting:
$\chi_i(g)=g(\d_i)$\,.

\noindent In particular $\boldsymbol{\chi}=(\chi_1,\ldots,\chi_m)$ defines a closed embedding
 \begin{equation*}
   \xymatrix{G\ \ar@{^(->}[r]^-{\boldsymbol{\chi}}&\T:=\Spec\K[\X^{\pm 1}]\cong(\K^*)^m}
 \end{equation*}
 the latter being the torus naturally acting by multiplication on $\overline{W}\cong\K^m$ and giving the obvious structure of toric variety to $\overline{W}$. In particular $\overline{i}$ is equivariant \wrt the $G$-actions on both $\overline{X}$ and $\overline{W}$. Recalling Proposition~\ref{prop:irr}~(3),  $\overline{i}(B_X)=\widetilde{B}$ guaranteeing that also $\widehat{i}:=\overline{i}|_{\widehat{X}}$ is a closed equivariant embedding. Passing to the quotient by $G$, we get then a well defined injection $i:X\hookrightarrow W:=\widehat{W}/G$, fitting into the commutative diagram (\ref{embediag}). In particular $i(X)$ is a closed subset of $W$, since $p_W^{-1}(W\backslash i(X))=\widehat{W}\backslash\widehat{i}(\widehat{X})$ is open.

 \noindent To show that $p_W:\widehat{W}\twoheadrightarrow W$ is a geometric quotient of toric varieties it suffices observing it is realizing the Cox quotient construction for $W$. In fact, by the exact sequence
  \begin{equation*}
   \xymatrix{1\ar[r]&G\ar[r]^-{\boldsymbol{\chi}}&\T\ar[r]&\Hom(K_0,\K^*)\ar[r]&1}
 \end{equation*}
 $W$ is naturally a toric variety under the action of the residue torus $\Hom(K_0,\K^*)$, with base point $x_0:=p_W([1:\cdots:1])$.  By the quotient construction one has the natural isomorphism $(p_W)^*:\cO_W\stackrel{\cong}{\longrightarrow} (p_W)_*\cO_{\widehat{W}}^G$. Moreover, $W$ is non-degenerate, that is it does not admit torus factors, as
 \begin{equation*}
   H^0(W,\cO_W^*)\stackrel{p_W^*}{\cong} H^0\left(W,((p_W)_*\cO_{\widehat{W}}^G)^*\right)=H^0\left(\widehat{W},(\cO_{\widehat{W}}^G)^*\right)\cong H^0\left(\widehat{W},\cO_{\widehat{W}}^*\right) \cong\K^*
 \end{equation*}
 by observing that $\widehat{W}$ does not admit any torus factor, as it is an open subset of $\overline{W}\cong\K^m$ whose complementary set $\widetilde{B}$ has codimension bigger than 2, and the $G$-action on constant functions is trivial. Finally, $W$ turns out to be $\Q$-factorial, as proved in the following Corollary~\ref{cor:WèQ-f}. We can then apply \cite[Thm.~2.1~(iii)]{Cox}. Then all the further properties of the geometric quotient $p_W:\widehat{W}\twoheadrightarrow W$ follow by the same arguments given when proving the previous Theorem~\ref{thm:K-emb}.
\end{proof}

\begin{remark}
  W{\l}odarzyck proved \cite[Thm.~A]{Wl}, that a $\Q$-factorial variety $X$ admits a closed embedding into a $\Q$-factorial toric variety if and only if it admits the following property
  \begin{itemize}
    \item[(V-2)] \emph{for any two points of $Y$ there exists an open affine set containing them}.
  \end{itemize}
  As a consequence of the previous Theorem~\ref{thm:canonic-emb}, one can then say that:
  \begin{itemize}
    \item \emph{a wMDS $X$ admitting a Cox basis $\X$ satisfies the (V-2) property.}
  \end{itemize}
The latter is proved, in a combinatorial language, in \cite[Thm.~3.2.1.4]{ADHL}, as the reader will understand after reading the next Remark~\ref{rem:bunched}.
\end{remark}

\subsubsection{Nomenclature (continued)}\label{sssez:nomenclatura2} We continue the nomenclature and notation introduced in \S~\ref{sssez:nomenclatura}.
\begin{enumerate}
  \item The ambient toric variety $W$, defined in Theorem~\ref{thm:canonic-emb}, only depends on the choices of the Cox basis $\X$ and no more on $K$ and $\chi$, as given in \ref{ssez:K} and \ref{ssez:chi}. In fact, for different choices $K',\chi'$ we get an isomorphic Cox ring, as observed in Remark~\ref{rems}~(1). Then it still admits the presentation $\K[\X]/I$, given in Remark~\ref{rems:wMDS}~(4), meaning that the toric embedding $i:X\hookrightarrow W$ remains unchanged, up to isomorphism. The latter is then called the \emph{$\X$-canonical toric embedding}, and $W$ is called the \emph{$\X$-canonical toric ambient variety}, of the wMDS $X$.

      \noindent Notice that the $\X$-canonical toric embedding exhibited in Theorem~\ref{thm:canonic-emb} only depends on the cardinality $|\X|$, which is fixed to be the minimum. Then, up to isomorphisms, $i:X\hookrightarrow W$ is then a \emph{canonical toric embedding}.
  \item Varieties $\widehat{W}$ and $\overline{W}$, exhibited in Theorem~\ref{thm:canonic-emb}, are called the \emph{characteristic space} and the \emph{total coordinate space}, respectively, of the canonical toric ambient variety $W$.
\end{enumerate}

\begin{remark}
  If $X$ is a $\Q$-factorial and complete toric variety, then the canonical toric embedding defined in \S~\ref{sssez:nomenclatura2}~(1) becomes trivial, giving $\overline{X}\cong\overline{W}$, $\widehat{X}\cong\widehat{W}$ and $X\cong W$.
\end{remark}

\subsubsection{The canonical toric embedding is a neat embedding}
Let $X$ be a wMDS and $i:X\hookrightarrow W$ be its $\X$-canonical toric embedding constructed in Theorem~\ref{thm:canonic-emb}, with $\X=\{x_1,\dots,x_m\}$ a given Cox basis of $\Cox(X)$. Recall notation introduced in Remark~\ref{rems}~(4) and in particular the definition of the fan matrix $V=(\v_1,\ldots,\v_m)$, which is a representative matrix of the dual morphism
\begin{equation*}
  \xymatrix{\Hom(\Div_\T(W),\Z)\ar[r]^-{div_W^\vee}_-V&N:=\Hom(M,\Z)}
\end{equation*}
whose columns are primitive generators of the rays in $\Si(1)$. In the following we will then denote
$D_i:=D_{\langle\v_i\rangle}$ the prime torus invariant associated with the ray $\langle\v_i\rangle\in\Si(1)$, for every $1\leq i\leq m$.

\begin{proposition}[Pulling back divisor classes]\label{prop:pullback} Let $i:X\hookrightarrow W$ be a closed embedding of a normal irreducible algebraic variety X into a toric variety $W(\Si)$ with acting torus $\T$. Let $D_{\rho}=\overline{\T\cdot x_{\rho}}$, for $\rho\in\Si(1)$, be the invariant prime divisors of $W$ and assume that $\{i^{-1}(D_\rho)\}_{\rho\in\Si(1)}$ is a set of pairwise distinct irreducible hypersurfaces in $X$. Then it is well defined a pull back homomorphism $$i^*:\Cl(W)\to\Cl(X)$$
\end{proposition}

\begin{proof}
  Let $x_0$ be the base point of $W$, as in Definition~\ref{def:TV}. Define
  \begin{equation*}
    W_\T:=\T\cdot x_0\cup \bigcup_{\rho\in\Si(1)}\T\cdot x_\rho\cap W_{\text{reg}}\,.
  \end{equation*}
  Our hypothesis on $\{i^{-1}(D_\rho)\}_{\rho\in\Si(1)}$ allows then us to conclude that $i^{-1}(W_\T)$ is a dense Zariski open subset of $X$. Let $D\in\Div_\T(W)$ be a torus invariant Weil divisor. Then $D\cap W_\T$ is a Cartier divisor on $W_\T\subseteq W_{\text{reg}}$ and define the pull back $i^\#(D):=\overline{i^\#(D\cap W_\T)}\in\Div(X)$ by pulling back local equations.
   This procedure clearly sends principal divisors to principal divisors, so defining a pull back homomorphism $i^*:\Cl(W)\to\Cl(X)$, by setting $i^*\d=[i^\#(d_W^{-1}(\d))]$, for every $\d\in\Cl(W)$ and recalling the surjection $d_W:\Div_\T(W)\twoheadrightarrow\Cl(W)$ presented in Remark~\ref{rems}~(4).
\end{proof}

\begin{definition}[Neat embedding]\label{def:neat}
  Let $X$ be a wMDS and $W(\Si)$ be a toric variety. Let $\{D_{\rho}\}_{\rho\in\Si(1)}$ be the torus invariant prime divisors of $W$. A closed embedding $i:X\hookrightarrow W$ is called a \emph{neat (toric) embedding} if
  \begin{itemize}
    \item[(i)] $\{i^{-1}(D_\rho)\}_{\rho\in\Si(1)}$ is a set of pairwise distinct irreducible hypersurfaces in $X$,
    \item[(ii)] the pullback homomorphism defined in Proposition~\ref{prop:pullback}, $$\xymatrix{i^*:\Cl(W)\ar[r]^-{\cong}&\Cl(X)}\,,$$
        is an isomorphism.
  \end{itemize}
\end{definition}

\begin{proposition}\label{prop:neat}
  The $\X$-canonical toric embedding $i:X\hookrightarrow W$, of a wMDS $X$ with Cox basis $\X$, is a neat embedding. Moreover the isomorphism $i^*:\Cl(W)\stackrel{\cong}{\longrightarrow}\Cl(X)$ restricts to give an isomorphism $\Pic(W)\cong\Pic(X)$.
\end{proposition}
\begin{proof}
  Given a Cox basis $\X=(x_1,\ldots,x_m)$, recall the construction of the $\X$-canonical ambient toric variety $W=W(\Si)$ given in Theorem \ref{thm:canonic-emb}. Consider the torus invariant prime divisor $D_i\in\Div_\T(W)$ associated with the ray $\langle\v_i\rangle\in\Si(1)$. Its pull back $\widehat{D}_i:=p_W^\#(D_i)$  is the principal divisor associated with $x_i|_{\widehat{W}}$, when $x_i$ is thought of as a homogeneous function on $\overline{W}$. Then $\widehat{i}^{-1}\left(\widehat{D}_i\right)=\left(x_i|_{\widehat{X}}\right)$\,, where $x_i$ is now thought of as a function on $\overline{X}$. Let $X_i\in K\leq\Div(X)$ and $\xi_i\in H^0(X,\cO_X(X_i))\subseteq\cS(X)$ be such that $\xi_i+\I_\chi(X)=\overline{x}_i:=x_i+I\in\Cox(X)$ and $\cD_i:=(\xi_i)+X_i$ is a prime divisor in $\Div(X)$. There is a unique choice for such a section $\xi_i$ and $p_X^\#(\cD_i)=\left(x_i|_{\widehat{X}}\right)$ \cite[\S~1.5.2, Prop.~1.5.3.5]{ADHL}.
  Then commutative diagram (\ref{embediag}) gives that $\cD_i=i^{-1}(D_i)$\,,
  meaning that hypotheses of Proposition~\ref{prop:pullback} are satisfied and there is a well defined pull back $i^*:\Cl(W)\to\Cl(X)$.

  \noindent To prove that $i$ is a neat embedding we have to show that $i^*$ is an isomorphism, that is $\cD_i\sim\cD_j$ if and only if $D_i\sim D_j$.
  In fact
  \begin{equation*}
    \cD_i\sim\cD_j \Longleftrightarrow  (\xi_i)+X_i\sim(\xi_j)+X_j\ \Longleftrightarrow\ X_i\sim X_j
  \end{equation*}
  Proposition~\ref{prop:Ichi} gives that $X_i\sim X_j$ if and only if  there exists an isomorphism
  $$\xymatrix{\psi:H^0(X,\cO_X(X_i))\ar[r]^{\cong}&H^0(X,\cO_X(X_j))}$$
  such that $\xi_i-\psi(\xi_i)\in\I_\chi(X)$, where $\psi(\xi_i)=\xi_i\cdot\chi(X_i-X_j)$. In particular
  \begin{equation*}
    (\psi(\xi_i))+X_j=(\xi_i)+X_i=\cD_i
  \end{equation*}
  The uniqueness of $\xi_j$ necessarily gives that $\psi(\xi_i)=k\xi_j$, for some $k\in\K^*$, and $\cD_j=\cD_i$. Then
  \begin{eqnarray*}
     X_i\sim X_j\ \Longleftrightarrow\ \xi_i-\xi_j\in\I_\chi(X)\
     \Longleftrightarrow\ x_i-x_j\in I
  \end{eqnarray*}
By minimality of $\X$  the latter can happen if and only if $x_i=x_j$, hence giving $D_i=D_j$.

 Since $i^*$ is an isomorphism, its restriction induces a monomorphism of $\Pic(W)\leq\Cl(W)$ into $\Pic(X)\leq\Cl(X)$. Then it suffices to show it is also onto $\Pic(X)$. Consider a class $[D]\in\Pic(X)$ with $D\in K$ a Cartier divisor, i.e. $D\in K\cap H^0(X,\mathcal{K}_X^*/\cO^*_X)$. Locally, on an open subset $U\subseteq X$, $D|_U$ is the principal divisor of a rational function $f/g\in \mathcal{K}^*_X(U)$, with $f,g\in\cO_X(U)$. Let $V\subseteq W$ be a Zariski open subset of $W$ such that $U=i^{-1}(V)$: then $\cO_X(U)=i_*\cO_X(V)$. Recalling the pull-back morphism $i^\#:\cO_W\longrightarrow i_*\cO_X$, locally defined by setting $(i|_U)^\#(h)=h\circ i|_U$ for every $h\in\cO_W(V)$, the thesis consists in showing that $f=(i|_U)^\#(\vf)$ and $g=(i|_U)^\#(\g)$, for some $\vf,\g\in\cO_W(V)$, so giving that $[D]=i^*[\cD]$ with $\cD|_V=(\vf/\g)$, that is $[\cD]\in\Pic(W)$.

 \noindent Notice that $\cO_W|_V\cong\cO_V$ and $i_*\cO_X|_V\cong (i|_U)_*\cO_U$, since $V$ and $U$ are open subsets of $W$ and $X$, respectively. Without lose of generality, assume $V$ is affine and consider the exact sequence of coherent $\cO_V$-modules
 \begin{equation*}
   \xymatrix{0\ar[r]&\mathfrak{I}\ar[r]&\cO_V\ar[r]^-{(i|_U)^\#}&(i|_U)_*\cO_U\ar[r]&0}
 \end{equation*}
 where $\mathfrak{I}$ is the sheaf of ideals of the closed embedding $i|_U:U\hookrightarrow V$ (for the coherence of $\mathfrak{I}$ and $(i|_U)_*\cO_U$ recall \cite[Prop.~II.5.9]{Hartshorne}). Then \cite[Prop.~II.5.6]{Hartshorne} gives that exactness is preserved on global sections, that is
 \begin{equation*}
   \xymatrix{0\ar[r]&\mathfrak{I}(V)\ar[r]&\cO_V(V)\ar[r]^-{(i|_U)^\#}&(i|_U)_*\cO_U(V)\ar[r]&0}
 \end{equation*}
 is an exact sequence: in particular $(i|_U)^\#:\cO_W(V)\twoheadrightarrow \cO_X(U)$ is surjective, so giving the thesis.
\end{proof}

\begin{corollary}\label{cor:WèQ-f}
  The canonical ambient toric variety $W$, of a wMDS $X$, is $\Q$-factorial.
\end{corollary}

\begin{proof}
  Recall that $X$ is $\Q$-factorial, then Proposition~\ref{prop:neat} gives
  $$\Pic(W)\otimes\Q\stackrel{i_\Q^*}{\cong}\Pic(X)\otimes\Q\cong\Cl(X)\otimes\Q\stackrel{(i_\Q^*)^{-1}}
  {\cong}\Cl(W)\otimes\Q\,.$$
\end{proof}

\begin{remark}
  The grading of $\Cox(W)\cong\K[\X]$ over $\Cl(W)$ is given by
  \begin{eqnarray*}
    \Cox(W)&=&\bigoplus_{\d\in\Cl(W)}\Cox(W)_\d\quad\text{with}\\
    \Cox(W)_\d&\cong&\K[\X]_\d:=\{x\in\K[\X]\,|\,\overline{x}=x+I\in\Cox(X)_{i^*(\d)}\}\,.
  \end{eqnarray*}
\end{remark}

\subsection{Bunch of cones} \label{ssez:Irr&bunches}

Let $X$ be a wMDS, $\X=\{x_1,\ldots,x_m\}$ be a Cox basis and $i:X\hookrightarrow W$ be the associated $\X$-canonical toric embedding. First of all notice that Proposition~\ref{prop:neat} implies that $X$ and $W$ have the same Picard number, given by $$r:=\rk\Cl(X)=\rk\Cl(W)$$
Moreover, recalling that there exist divisors $X_1,\dots,X_m\in K$ such that
$$ \forall\,i=1,\ldots,m\quad x_i+I=\xi_i+\I_\chi(X)\quad\text{and}\quad\xi_i\in H^0(X,\cO(X_i))$$
as exhibited in the proof of Proposition~\ref{prop:neat}, up to shrink $K$, one can assume
\begin{equation}\label{K-shrink}
  K=\Ls(X_1,\ldots,X_m)
\end{equation}
In fact, calling $\d_i=d_K(X_i)$, for every $D\in K$
\begin{equation*}
  d_K(D)\in\Cl(X)=\Ls(\d_1,\ldots,\d_m)\ \Longrightarrow\ [D]=\sum_i a_i\d_i\ \Longrightarrow\ D\sim \sum_i a_i X_i\,.
\end{equation*}
 Let us then set
\begin{equation*}
  \rk\Div_{\T}(W)=\rk K=m\quad,\quad n:=m-r>0\,.
\end{equation*}
Choose generators $d_1,\ldots d_r$, $\ve_1,\ldots,\ve_s$ for $\Cl(W)$ such that
\begin{equation*}
\Cl(X)\stackrel{(i^*)^{-1}}{\cong}\Cl(W)\cong F\oplus\Tors\cong\left(\bigoplus_{i=1}^r\Z\cdot d_i\right)\oplus\left(\bigoplus_{k=1}^s\Z\cdot\ve_i\right)\cong\Z^r\oplus\left(\bigoplus_{k=1}^s\Z/\tau_i\Z\right)\,.
\end{equation*}
Calling $d_W:\Div_\T(W)\twoheadrightarrow\Cl(W)$ the surjection defined by the exact sequence (\ref{Tdivisori}), there are induced decompositions $d_W=f_W\oplus \tau_W$ and $d_K=f_K\oplus \tau_K$ in free parts $f_W,f_K$ and torsion parts $\tau_W,\tau_K$. Consider the weight and torsion matrices defined in Remark~\ref{rems}~(4), that is
\begin{itemize}
  \item $Q=(\q_1\,\ldots\,\q_m)$ being a $r\times m$ weight matrix representing $f_W$ on the bases $\{D_j\}$ and $\{d_i\}$,
  \item $\Ga=(\boldsymbol\g_1\,\ldots\,\boldsymbol\g_m)$ being a $s\times m$ torsion matrix representing $\tau_W$ on the bases $\{D_j\}$ and $\{\ve_k\}$, where the $k$-th entry of $\boldsymbol\g_j$ is a class $\g_{kj}\in\Z/\tau_i\Z$.
\end{itemize}
The situation is then described by the following commutative diagram
 \begin{equation}\label{div-diagram-embedding}
      \xymatrix{0 \ar[r]&  M \ar[r]^-{div_W}_-{V^T}\ar[d]_-{\mathbf{I}_{n}} &
\Div_\T(W) \ar[r]^-{d_W}_-{Q\oplus\Ga}\ar[d]^-{i^\#}_-{\mathbf{I}_{m}} & \Cl(W) \ar[d]^-{i^*}_-{\mathbf{I}_{r}\oplus\overline{\mathbf{I}}_s} \ar[r]&0\\
0 \ar[r]& K_0 \ar[r]_-{V^T}&K\ar[r]^-{d_{K}}_-{Q\oplus\Ga} & \Cl (X) \ar[r]& 0 \\
 }
\end{equation}
The \emph{pseudo-effective cone} $\overline{\Eff}(W)$ is then given by the cone generated by the columns of $Q$ i.e.
$$\overline{\Eff}(W)=\langle Q\rangle\subseteq\Cl(W)\otimes\R=N^1(W)$$
It is well known that $\overline{\Eff}(W)$ supports a fan called the \emph{secondary fan} or \emph{GKZ subdivision} (see e.g. \cite[\S~15]{CLS}).

Consider the irrelevant ideal $\irr(X)$, defined in Definition~\ref{def:irr}, and let $f$ be a homogeneous generator of $\irr(X)$. Since $\X$ is a Cox basis, $f$ can be expressed as a product of $\overline{x}_1,\ldots,\overline{x}_m$, up to a constant in $\K^*$ (see e.g. the proof of \cite[Thm.~1.5.3.7]{ADHL}), meaning that $\widetilde{\irr}=\pi_\X^{-1}(\irr(X))$ is a homogeneous monomial ideal in $\K[\X]$.
On the other hand, one has an explicit description of the irrelevant ideal of $W$, given by
\begin{eqnarray}\label{Irr}
\nonumber
  \I rr(W)&=&\widetilde{\I rr}=\left(\prod_{\rho\not\in\s}x_\rho\,|\,\s\in\Si\right)\\
  &=&\left(\prod_{\rho\not\in\s}x_\rho\,|\,\s\in\Si_{\max}\right)\subseteq \K[\X]\cong\Cox(W)
\end{eqnarray}
where $\Si_{\max}$ is the collection of maximal cones of $\Si$ (see e.g. \cite{Cox}).

\begin{remark}\label{rem:corrispondenze}
Notice that:
\begin{itemize}
  \item[(a)] generators of $\I rr(W)$ are in one to one correspondence with cones of the \emph{bunch} $\B$, associated via Gale duality with the fan $\Si$ of $W$ (see \cite{Berchtold-Hausen}, \cite[\S~2.2.1]{ADHL}); recalling notation introduced in \ref{sssez:lista}, this correspondence can be made more explicit as follows
\begin{equation}\label{irr vs bunch}
  \I rr(W)=\left(\prod_{i\not\in I}x_i\,|\,I\in\I_\Si\right)\,\longleftrightarrow\,\B=\left\{\langle Q^I\rangle\,|\,I\in\I_\Si\right\}
\end{equation}
  \item[(b)] the common intersection of all cones in $\B$ gives the cone $\Nef(W)$, generated by the classes of nef divisors inside $\overline{\Eff}(W)$ \cite[Thm.~10.2]{Berchtold-Hausen}, that is
      \begin{equation*}
  \Nef(W)=\bigcap_{I\in\I_\Si}\langle Q^I\rangle=\bigcap_{I\in\I_{\Si_{\max}}}\langle Q^I\rangle
\end{equation*}
where the last equality comes from (\ref{Irr}) and (\ref{irr vs bunch});
\item[(c)] let $\overline{\Mov}(W)$ be the closed cone defined by classes of \emph{movable} divisors of $W$ \cite[(15.1.7)]{CLS}; then $\overline{\Mov}(W)=\bigcap_{i=1}^m\langle Q^{\{i\}}\rangle$ \cite[Prop.~15.2.4]{CLS}, so giving, by the previous item (b),
      $$ \Nef(W)\subseteq\overline{\Mov}(W)\subseteq\overline{\Eff}(W)\,.$$
\end{itemize}
\end{remark}

\begin{remark}\label{rem:bunched}
  Keeping in mind diagram (\ref{div-diagram-embedding}) and Proposition~\ref{prop:irr}, correspondence (\ref{irr vs bunch}) shows that a wMDS $X$ is actually a \emph{variety arising from a bunched ring}, in the sense of \cite[Construction~3.2.1.3]{ADHL}. In particular the bunched ring associated with $X$ is the triple $(\Cox(X),\X,\B)$.
\end{remark}

\begin{proposition}\label{prop:Qfactorial}
Let $X$ be a wMDS, $\X$ be a Cox basis and $i:X\hookrightarrow W$ be the associated $\X$-canonical toric embedding. Then the isomorphism $i^*$, between class groups, extends to give an isomorphism $i^*_\R:N^1(W)\stackrel{\cong}{\longrightarrow}N^1(X)$ such that
\begin{equation*}
  i^*_\R\left(\overline{\Eff}(W)\right)=\overline{\Eff}(X)\quad,\quad i^*_\R\left(\overline{\Mov}(W)\right)=\overline{\Mov}(X)\quad,\quad
  i^*_\R\left(\Nef(W)\right)=\Nef(X)\,.
\end{equation*}
\end{proposition}

\begin{proof} The statement is a direct consequence of diagram (\ref{div-diagram-embedding}). For any further detail, the interested reader is referred to \cite[\S~3.3.2]{ADHL} (in particular Prop.~3.3.2.9 and Prop.~3.2.5.4(v)), keeping in mind that a wMDS is a variety arising from a bunched ring, as just observed in Remark~\ref{rem:bunched}.
\end{proof}

\begin{corollary}\label{cor:Qpositive}
  Let $X$ be a complete wMDS endowed with a Cox basis $\X$ and $i:X\hookrightarrow W$ be its $\X$-canonical toric embedding. Then $\overline{\Eff}(X)\cong\overline{\Eff}(W)$ are strongly convex, full dimensional cones in $N^1(X)\cong N^1(W)$, respectively. In particular the weight matrix $Q$ in diagram (\ref{div-diagram-embedding}) turns out to be a $W$-matrix and consequently the fan matrix $V$ is a $F$-matrix. Moreover $Q$ can then be given by a positive matrix, meaning that both $\overline{\Eff}(X)\cong\overline{\Eff}(W)$ can be thought of as subcones of the positive orthant in $N^1(X)\cong N^1(W)$, respectively.
\end{corollary}

\begin{proof}
$X$ is complete and irreducible meaning that $H^0(X,\cO_X)\cong\K$. Therefore an effective and principal divisor is necessarily zero. Assume that $[D],[D']\in\overline{\Eff}(X)$ with $D,D'$ effective divisors and $[D]+[D']=[D+D']=0$. Then $D+D'$ is an effective and principal divisor, that is $D+D'=0$. But $D$ and $D'$ are effective, then $D=D'=0$. This suffices to show that $\overline{\Eff}(X)$ is strongly convex. Clearly $\overline{\Eff}(X)$ is full dimensional since $Q$ has maximal rank $r=\dim N^1(X)$.

Recall Definition \ref{def:Wmatrice} of a W-matrix. We just proved condition (a). For (c), up to a linear automorphism of $N^1(W)$, one can always assume $\overline{\Eff}(W)=\langle Q\rangle$ contained in the positive orthant of $N^1(W)$, proving also the last part of the statement. For (b), dualize the upper exact sequence in diagram (\ref{div-diagram-embedding}), getting the following exact sequence
\begin{equation*}
  \xymatrix{0\ar[r]&\Hom(\Cl(W),\Z)\ar[r]^-{d_W^\vee}_-{Q^T}&\Hom(\Div_\T(W),\Z)\ar[r]^-{div_W^\vee}_-{V}&N}
\end{equation*}
in which $\Ls_r(Q)\cong\ker div_W^\vee\leq\Hom(\Div_\T(W),\Z)\cong\Z^m$. Then $\Ls_r(Q)$ is a free subgroup of $\Z^m$. For (d), a zero column in $Q$ means that the corresponding torus invariant divisor $D_\rho$ is principal, which cannot happen. Last conditions (e) and (f) follow by observing that $Q$ is a Gale dual matrix of the fan matrix $V$ of $W$: those conditions correspond to say that the zero vector can't generate a ray of $\Si$ and distinct rays in $\Si$ can't be generated by proportional vectors, respectively.

Finally $V$ turns out to be a $F$-matrix by \cite[Prop.~3.12.2]{RT-LA&GD}.
\end{proof}

\subsection{The GKZ-decomposition: cones, cells and chambers}\label{ssez:GKZ}
Recall that, given a toric variety $W$, both $\overline{\Eff}(W)$ and its subcone $\overline{\Mov}(W)$ support a fan structure, the so called \emph{Gelfand-Kapranov-Zelevinsky (GKZ)-decomposition} or \emph{secondary fan} (see \cite[\S\,14.4]{CLS} and references therein). For quickly visualize GKZ-cones composing such a fan, consider \cite[Def.\,1.7]{RT-Qfproj}. Here we will use the interpretation of GKZ-cones given in \cite[\S\,2.2.2]{ADHL}, to which the interested reader is referred for any further detail.

\subsubsection{Notation} Let $W$ be a $n$-dimensional toric variety of Picard number $r$ and $Q$ be a weight matrix of $W$. Then $\overline{\Eff}(W)=\langle Q\rangle$ and $\overline{\Mov}(W)=\bigcap_{i=1}^m\langle Q^{\{i\}}\rangle$, as observed in the previous section \S~\ref{ssez:Irr&bunches}. For this reason, in the following we will often denote the polyhedral cones $\langle Q\rangle$ and $\bigcap_{i=1}^m\langle Q^{\{i\}}\rangle$ by $\Eff(Q)$ and $\Mov(Q)$, respectively.

\subsubsection{Assumption}
In the following we will always assume that $\Mov(Q)$ is of full dimension $r=\rk(Q)$ inside $\Eff(Q)$. This is certainly the case when $Q$ is a weight matrix of the canonical ambient toric variety of a wMDS \cite[Thm.\,2.2.2.6\,(i)]{ADHL}.

 \begin{definition}[GKZ-cones and decomposition]\label{def:GKZ} In the above notation, for every $w\in\langle Q\rangle$, define the \emph{GKZ-cone} associated with $w$ as the following polyhedral cone
 \begin{eqnarray*}
 \g_w&:=&\bigcap\{\langle Q^I \rangle\,|\,\forall\,I\subseteq\{1,\ldots,n+r\}\,:\,w\in\langle Q^I \rangle\}\\
 &=&\bigcap\{\langle Q^J \rangle\,|\,\forall\,J\subseteq\{1,\ldots,n+r\}\,:\,w\in\Relint\langle Q^J \rangle\}
 \end{eqnarray*}
 The \emph{GKZ-de\-com\-po\-si\-tion} of $\Eff(Q)$ is the collection $\Ga(Q)= \{\g_w\,|\,w\in\langle Q\rangle\}$ of GKZ-cones. The same construction on $\Mov(Q)$ defines the GKZ-decomposition $\Ga(Q)|_{\Mov}$ of $\Mov(Q)$.
 \end{definition}

 \begin{remark}[Construction 2.2.2.1 and Thm.~2.2.2.2 in \cite{ADHL}] Notice that:
 \begin{itemize}
   \item[(a)] if $w\in\langle Q\rangle$ then $w\in\Relint(\g_w)$;
   \item[(b)] if $\g\in\Ga(Q)$ then $\g=\g_w$, for every $w\in\Relint(\g)$;
   \item[(c)] $\Ga(Q)$ and $\Ga(Q)|_{\Mov}$ turn out to be fans with support $\Eff(Q)=\langle Q\rangle$ and $\Mov(Q)=\bigcap_{i=1}^m\langle Q^{\{i\}}\rangle$, respectively.
 \end{itemize}
    \end{remark}

With every cone $\g\in\Ga(Q)$ one can associate the following bunch of cones
 \begin{equation*}
   \B_{\g}:=\{\langle Q^I\rangle\,|\,I\subseteq\{1,\ldots,n+r\}\,:\,\Relint(\g)\subseteq\Relint\langle Q^I\rangle\}\,.
 \end{equation*}
Let $V$ be a Gale dual matrix of $Q$ giving a fan matrix of $W$. Then $\B_{\g}$ determines, by Gale duality, the following collection of convex cones
\begin{equation}\label{Si_g}
\Si_{\g}:=\{\langle V_I\rangle\,|\,\langle Q^I\rangle\in\B_{\g}\}
\end{equation}
which in general is a \emph{quasi-fan}, meaning that it may admit a cone containing a line.

\begin{proposition}\label{prop:quasiproiettivo} Given a GKZ-cone $\g\in\Ga(Q)$ the following facts hold:
\begin{enumerate}
  \item the support $|\Si_\g|=\langle V\rangle$ \cite[Prop.~2.2.4.1]{ADHL};
  \item $\Si_\g$ is a fan if and only if $\Relint(\g)\subseteq\Relint(\Mov(Q))$ \cite[Thm.~2.2.2.6\,(ii)]{ADHL};
  \item if $\Si_\g$ is a fan then the associated toric variety $X(\Si_\g)$ is quasi-projective \cite[Thm.~2.2.2.2\,(ii)]{ADHL};
  \item if $\Si_\g$ is a fan then $\g$ is full-dimensional inside $\Mov(Q)$ if and only if $\Si_\g$ is simplicial \cite[Thm.~2.2.2.6\,(iii)]{ADHL}. In particular $\Si_\g\in\SF(V)$\,.
\end{enumerate}
\end{proposition}

\begin{definition}[Geometric cells and chambers]\label{def:cell}
A GKZ-cone $\g\in\Ga(Q)|_{\Mov}$ is called a \emph{geometric cell} (or simply a \emph{g-cell}) if there exists a fan $\Si\in\SF(V)$ such that $\g=\g_\Si:=\bigcap_{I\in\I_\Si}\langle Q^I\rangle$.  A full dimensional geometric cell inside $\Mov(Q)$ is called a \emph{geometric chamber} (or simply a \emph{g-chamber}).
\end{definition}

Recalling Remark\,\ref{rem:corrispondenze}\,(b), if $\langle V\rangle=\R^n$ then a GKZ-cone $\g$ is  a g-cell if it is the $\Nef$ cone of a $\Q$-factorial and complete toric variety: this is the case e.g. when $V$ is a fan matrix of the canonical ambient toric variety $W$ of a complete wMDS $X$.

\noindent In general a GKZ-cone may not be a g-cell: e.g., as a consequence of considerations we will give in the following \S\,\ref{ssez:proj-emb}, if $\rk(Q)=2$ then every 1-dimensional GKZ-cone cannot be a g-cell, since every $\Q$-factorial complete toric variety of Picard number 2 is projective, thus it admits a full-dimensional $\Nef$ cone \cite[Thm.\,3.2]{RT-r2proj}.

\begin{proposition}\label{prop:cell=camera}
  A geometric cell $\g_{\Si}$ is a geometric chamber if and only if $\Si_{\g_{\Si}}=\Si$.
\end{proposition}

\begin{proof}
  Assume that $\g:=\g_\Si=\bigcap_{I\in\I_\Si}\langle Q^I\rangle$ is a g-chamber. Then $$\forall\,I\in\I_\Si\quad\Relint(\g)\subseteq\Relint\langle Q^I\rangle\ \Longrightarrow\ \B\subseteq\B_\g$$
  where $\B$ is the bunch of cones associated with $\Si$ by Gale duality. This means that $\Si\subseteq\Si_\g$, that is $\Si=\Si_\g$, as $\Si\in\SF(V)$.

  Conversely if $\Si=\Si_\g$ then $\B=\B_\g$. Since $\Si$ is simplicial, Proposition\,\ref{prop:quasiproiettivo}\,(4) guarantees that $\g$ is full dimensional.
\end{proof}

\begin{proposition}\label{prop:camera}
   Every full dimensional GKZ-cone $\g\subseteq\Mov(Q)\cong\overline{\Mov}(W)\cong\overline{\Mov}(X)$ is a geometric chamber.
  \end{proposition}

 \begin{proof}
For every $w\in \Relint\g$, one has $\g_w=\g$, meaning that every cone composing the associated bunch $\B_\g$ is full dimensional. Therefore the associated fan $\Si_\g$ is simplicial. By \cite[Prop.~2.2.4.1]{ADHL} the support of $\Si_\g$ is given by $\langle V\rangle$ and the latter fills up the whole $\R^n$. Then $\Si_\g\in\SF(V)$ and $\g=\bigcap_{I\in\I_{\Si_\g}}\langle Q^I\rangle$ is a g-chamber.
 \end{proof}

\begin{proposition}\label{prop:nefècella}
The $\Nef$ cone of a  $\Q$-factorial and complete toric variety is always a geometric cell.
\end{proposition}

\begin{proof}
Let $Z=Z(\Si)$ be  a $\Q$-factorial and complete toric variety and let $Q$ be a  weight matrix of $Z$. This means that $\Si\in\SF(V)$, where $V$ is a fan matrix of $Z$. By Remark\,\ref{rem:corrispondenze}\,(b)
$$\Nef(Z)=\bigcap_{I\in\I_\Si}\langle Q^I\rangle=\bigcap\{\langle Q^I\rangle\,|\,\langle Q^I\rangle\in\B\}\subseteq\Mov(Q)$$
where $\B$ is the bunch of cones associated with $\Si$.
Then $\Nef(Z)$ is a finite union of GKZ-cones in $\Ga(Q)|_{\Mov}$. Let us first of all assume that $\Nef(Z)$ is full dimensional inside $\Mov(Q)$. Choose $w\in\Relint(\Nef(Z))$ and consider the associated GKZ-cone $\g_{w}$. Then $\Relint(\g_w)\subseteq\Relint(\langle Q^I\rangle)$, for every $I\in\I_\Si$, so giving that $\B\subseteq\B_{\g_w}$. Then $\Si\subseteq\Si_{\g_w}$, meaning that $\Si_{\g_w}=\Si$, as $\Si\in\SF(V)$. Then $\g_w=\Nef(Z)$, for every $w\in\Relint(\Nef(Z))$, so giving that $\Nef(Z)$ is a g-cell.

Assume now that $\Nef(Z)$ is not full dimensional in $\Mov(Q)$ and choose $w\in\Relint(\Nef(Z))$. Then Proposition\,\ref{prop:quasiproiettivo}~(1) gives that the fan $|\Si_{\g_w}|=\langle V\rangle$. Let $\Si'$ be a  simplicial refinement of $\Si_{\g_w}$. Then $\Si'\in\SF(V)$ and Proposition \ref{prop:quasiproiettivo} implies that $\g':=\bigcap_{I\in\I_{\Si'}}\langle Q^I\rangle$ is full dimensional in $\Mov(Q)$. By the previous part of the proof, $\g'$ turns out to be a g-chamber. Since $\Si'$ is a refinement  of $\Si_{\g_w}$, one has that $\g_w$ is a face of $\g'$, for every $w\in\Relint(\Nef(Z))$, by \cite[Thm.~2.2.2.2]{ADHL}. Then $\Nef(Z)$ is a face of $\g'$, as it is a convex cone. Finally, the fan structure of the GKZ-decomposition allows us to concluding that $\Nef(Z)$ a is actually a GKZ-cone, hence a g-cell.
\end{proof}

\subsection{Sharp completions of the canonical ambient toric variety}\label{ssez:sharp}

As already observed in  the introduction,  by Nagata's theorem  e\-ve\-ry algebraic variety can be embedded in a complete one and, for  those endowed with an algebraic group action Sumihiro provided an equivariant version of this theorem. In particular, for toric varieties, it corresponds with the Ewald-Ishida  combinatorial completion procedure for fans, recently simplified by Rohrer. Anyway, all these procedures in general require the adjunction of some new ray into the fan under completion, that is an increasing of the Picard number.

In the following we  call \emph{sharp} a completion which does not increase the Picard number. Although  a sharp completion of a toric variety does not exist in general,  Hu and Keel showed that the canonical ambient toric variety $W$ of a MDS $X$ admits sharp completions, which are even projective, one for each g-chamber contained in $\Nef(W)\cong\Nef(X)$ \cite[Prop.~2.11]{Hu-Keel}.

Unfortunately, this Hu-Keel result does no more hold for the ambient toric variety of a wMDS. The key condition one needs to impose, in order to guarantee the existence of a sharp completion in the non-projective set up, is the existence of particular cells inside the $\Nef$ cone, in the following called  \emph{filling cells}.

 Let us  then set some notation.

\begin{definition}[Filling cell]\label{def:filling} Let $X$ be a wMDS endowed with a Cox basis $\X$ and $i:X\hookrightarrow W$ its $\X$-canonical toric embedding. A geometric cell $\g\subseteq\Nef(X)\cong\Nef(W)$ is called \emph{filling} if $\g=\bigcap_{I\in\I_{\Si'}}\langle Q^I\rangle$ for some fan $\Si'\in\SF(V)$ containing the fan $\Si$ of $W$, that is $\Si\subseteq\Si'$. The associated fan $\Si'$ is called a \emph{filling fan}. Moreover, whenever $W$ is smooth, a  filling cell is called \emph{smooth} if it is associated to a regular filling fan $\Si'$.
\end{definition}

\begin{definition}[Fillable wMDS]\label{def:fillable} A wMDS $X$ is called \emph{fillable} if $\Nef(X)$ contains a filling cell $\g$. Moreover, it is called \emph{smoothly fillable} if $\g\subseteq\Nef(X)$ is a smooth filling cell.
\end{definition}
In general a g-cell may correspond to more than one fan and may not be a filling cell as the following example shows.

\begin{example}\label{ex:cells}
Consider the Berchtold-Hausen example \cite[Ex.~10.2]{Berchtold-Hausen}. Here we will refer to the analysis performed in \cite[Rem.~3.1]{RT-r2proj}. Let $X=W$ be the MDS given by the $\Q$-factorial complete toric variety whose fan matrix $V$ and weight matrix $Q$ are given by
\begin{figure}
\begin{center}
\includegraphics[width=8truecm]{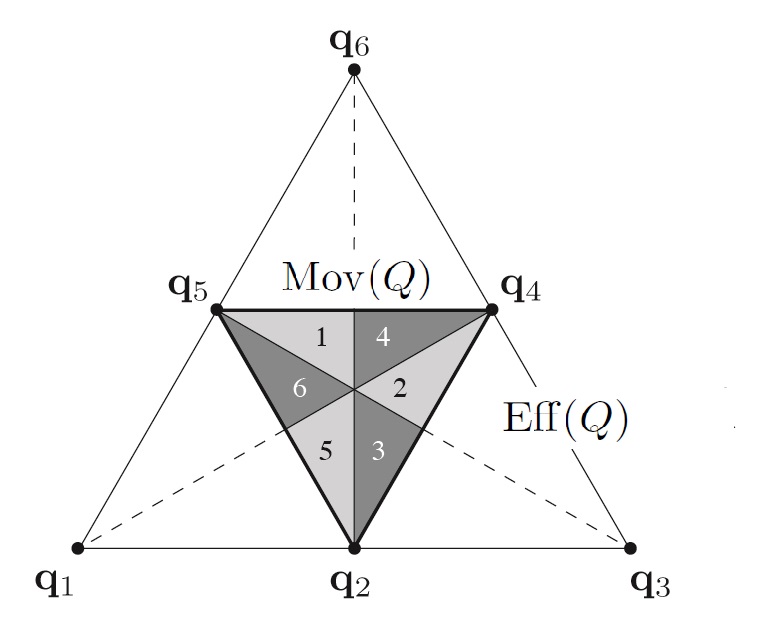}
\caption{\label{Fig1}The section of the cone $\Eff(Q)$ in Example~\ref{ex:cells}, which is the positive orthant of $\R^3$, with the plane $x_1+x_2+x_3=1$.}
\end{center}
\end{figure}
\begin{equation*}
V:= \left(
  \begin{array}{cccccc}
    1&0&0&0&-1&1 \\
    0&1&0&-1&-1&2 \\
    0&0&1&-1&0&1 \\
  \end{array}
\right)
\ \Rightarrow\
Q=\left(
     \begin{array}{cccccc}
       1&1&0&0&1&0\\
      0&1&1&1&0&0\\
       0&0&0&1&1&1 \\
     \end{array}
   \right)
\end{equation*}
and whose fan $\Si_1:=\Si_{\g_1}$ is the fan determined by  the g-chamber 1 in Fig.~\ref{Fig1}. The intersection of all the six chambers contained in $\Mov(Q)$ is given by the cone ge\-ne\-ra\-ted by the anti-canonical class $\g=\left\langle\begin{array}{c}
1\\
1\\
1
\end{array} \right\rangle$ (see Fig.~\ref{Fig1}). It is clearly a GKZ-cone contained in the boundary $\partial \g_1$. In particular $\g$ is the g-cell corresponding to both the two different fans $\Si,\Si'\in\SF(V)$ giving rise to complete and non-projective $\Q$-factorial toric varieties, as  described in \cite[Rem.~3.1]{RT-r2proj}. Then $\g$ is a g-cell inside $\g_1=\Nef(W)$ but it cannot be a filling cell since both $\Si,\Si'$ do not contain  $\Si_1$, and any further g-cell described by the remaining fans in $\SF(V)$ are given by the chambers  determined by numbers from 2  to 6  in Fig.~\ref{Fig1}.
\end{example}

Anyway what described in the previous Example~\ref{ex:cells} cannot occur for full dimensional cells, as proved by the following

\begin{proposition}\label{prop:camera-filling}
A geometric chamber $\g\subseteq\Nef(X)$ is always a filling cell. Moreover the filling fan determined by $ \g$ is unique and given by $ \Si_\g$, as defined in (\ref{Si_g}).
\end{proposition}

\begin{proof} Since $\g$ is a g-chamber, $\Relint(\g)\subseteq \Relint(\g_\Si)$, where $\g_\Si=\Nef(Z)$. Then $\B\subseteq \B_\g$, that is $\Si\subseteq\Si_\g$.

For what is concerning the uniqueness of the filling fan associated with $\g$, assume that $\Si'$ is a further filling fan associated with $\g$, that is $\g=\bigcap_{I\in\I_{\Si'}}\langle Q^I\rangle$. Let $\B'$ be the associated bunch of cones. For every $w\in\Relint(\g)$  one has
$$\forall\,I\in\I_{\Si_\g}\,,\,J\in\I_{\Si'}\quad w\in\langle Q^I\rangle\cap\langle Q^J\rangle\ \Longrightarrow\ \Relint\langle Q^I\rangle\cap\Relint\langle Q^J\rangle\neq\emptyset$$
so giving that $\B'=\B_\g$, as they are maximal bunches. Then $\Si'=\Si_\g$.
\end{proof}

\begin{remark}
Putting together what observed until now and assuming $\langle V\rangle=\R^n$, we get the following picture
\begin{equation*}
  \xymatrix{\SF(V)\ar@{->>}[r]^-{\nu}&\{\text{g-cells}\}\ar@{^(->}[r]^-{\eta}&\Ga(Q)|_{\Mov}\\
            \P\SF(V)\ar@{^(->}[u]\ar[r]_-{1:1}^-{\nu_\P}&\{\text{g-chambers}\}\ar@{^(->}[u]
            \ar[r]_-{1:1}^-{\eta_\P}&
            \Ga(Q)|_{\Mov}(r)\ar@{^(->}[u]}
\end{equation*}
where
\begin{itemize}
  \item the map $\nu$ is defined by setting $\nu(\Si):=\Nef(Z(\Si))$, as described in Proposition~\ref{prop:nefècella}: Example~\ref{ex:cells} shows that $\nu$ is not injective. But it is surjective by Definition~\ref{def:cell} of a g-cell;
  \item $\P\SF(V)=\{\Si\in\SF(V)\,|\,Z(\Si)\ \text{is (quasi-)projective}\}$;
  \item the map $\nu_\P$ is the restriction of $\nu$ to the subset $\P\SF(V)\subseteq\SF(V)$: it gives a bijection on the subset of g-chambers by Propositions~\ref{prop:quasiproiettivo} and \ref{prop:camera-filling};
  \item the map $\eta$ is given by the definition \ref{def:cell} of a g-cell; its restriction to the subset of chambers gives rise to a bijection $\eta_\P$ on the $r$-skeleton of the GKZ-decomposition by Proposition~\ref{prop:camera}, assuming $r:=\rk(Q)$.
\end{itemize}
\end{remark}

The following is an attempt of extending, to the non-projective setup, under suitable conditions, the Hu-Keel sharp completion (see items (1), (2) and (3) in \cite[Prop.~2.11]{Hu-Keel}). A comparison with Construction~3.2.5.6 in \cite{ADHL}, may be useful.

\begin{theorem}\label{thm:complete-emb} Let $X$ be a fillable wMDS with Cox basis $\X$, $i:X\hookrightarrow W$ its $\X$-canonical toric embedding and $V$ a fan matrix of $W$. Then the following facts hold.
\begin{enumerate}
  \item For any filling cell $\g\subseteq\Nef(W)$, there esists a $\Q$-factorial partial completion $Z$ of $W$, with an associated open embedding $\iota:W\hookrightarrow Z$, such that:
       \begin{itemize}
         \item the composed embedding $\iota\circ i: X\hookrightarrow Z$ induces a neat embedding of the closure of $X$ inside $Z$, that is $j:X':=\overline{\iota\circ i(X)}\hookrightarrow Z$; in particular $X'$ is still a wMDS giving a partial completion of $X$;
         \item $Z$ is complete if and only if $\langle V\rangle=\R^n$, that is if and only if $V$ is an $F$-matrix: in this case $Z$ is a sharp completion of $W$ and $X'$ is a complete wMDS giving a sharp completion of $X$;
         \item if $X$ is complete then $X'\cong X$ and $\iota\circ i:X\hookrightarrow Z$ is a neat embedding;
         \item there is a commutative dia\-gram between total coordinate spaces, characteristic spaces and toric embeddings, given by
      \begin{equation*}
        \xymatrix{\overline{X}\ar@{^{(}->}[r]^-{\overline{i}}&\overline{W}\ar[r]^-{id}&\overline{Z}\\
                    \widehat{X}\ar@{->>}[d]^-{p_X}\ar@{^{(}->}[u]^-{j_X}\ar@{^{(}->}[r]^-{\widehat{i}}&
                    \widehat{W}\ar@{->>}[d]^-{p_W}\ar@{^{(}->}[u]^-{j_W}
                    \ar@{^{(}->}[r]^-{\widehat{\iota}}&\widehat{Z}\ar@{->>}[d]^-{p_Z}\ar@{^{(}->}[u]^-{j_Z}\\
                    X\ar[rrd]^{\iota\circ i}\ar@{^{(}->}[r]^-{i}&W\ar@{^{(}->}[r]^-{\iota}&Z\\
                    &&X'\ar@{^{(}->}[u]^-j}
      \end{equation*}
      In particular $p_Z:\widehat{Z}\twoheadrightarrow Z$ is a 1-free geometric quotient coming from the action of the characteristic quasi-torus $G$ on $\overline{Z}$ and $j_Z(\widehat{Z})=\overline{Z}^{s}=\overline{Z}^{ss}$.
       \end{itemize}
  \item There is an induced diagram of group isomorphisms
      \begin{equation*}
        \xymatrix{\Cl(Z)\ar[r]^-{\iota^*}_-{\cong}&\Cl(W)\ar[r]^-{i^*}_-{\cong}&\Cl(X)\\
                    \Pic(Z)\ar@{^{(}->}[u]\ar[r]^-{\iota^*}_-{\cong}&\Pic(W)\ar@{^{(}->}[u]
                    \ar[r]^-{i^*}_-{\cong}&\Pic(X)\ar@{^{(}->}[u]}
      \end{equation*}
  \item Isomorphisms $\iota^*$ and $i^*$ extend to give $\R$-linear isomorphisms
      \begin{equation*}
        \xymatrix{N^1(Z)\ar[r]^-{\iota^*_\R}_-{\cong}&N^1(W)\ar[r]^-{i^*_\R}_-{\cong}&N^1(X)\\
                    \overline{\Eff}(Z)\ar@{^{(}->}[u]\ar[r]^-{\iota^*_\R}_-{\cong}&\overline{\Eff}(W)\ar@{^{(}->}[u]
                    \ar[r]^-{i^*_\R}_-{\cong}&\overline{\Eff}(X)\ar@{^{(}->}[u]\\
                    \overline{\Mov}(Z)\ar@{^{(}->}[u]\ar[r]^-{\iota^*_\R}_-{\cong}&\overline{\Mov}(W)\ar@{^{(}->}[u]
                    \ar[r]^-{i^*_\R}_-{\cong}&\overline{\Mov}(X)\ar@{^{(}->}[u]\\
                    \g=\Nef(Z)\ar@{^{(}->}[u]\ar@{^{(}->}[r]^-{\iota^*_\R}&\Nef(W)\ar@{^{(}->}[u]
                    \ar[r]^-{i^*_\R}_-{\cong}&\Nef(X)\ar@{^{(}->}[u]}
      \end{equation*}
\end{enumerate}
\end{theorem}

\begin{proof}\hfill\newline
(1):  Let $Q=\G(V)$ be a weight matrix of $W$, $\Si$ be the fan of $W$ and $\B$ be the associated bunch of cones. Since $X$ is fillable, there exists a filling cell $\g\subseteq\Nef(W)\cong\Nef(X)$, that is, there exists $\Si'\in\SF(V)$ such that $\g=\bigcap_{I\in\I_{\Si'}}\langle Q^I\rangle$ and $\Si\subseteq\Si'$.  Let $Z=Z(\g,\Si')$ be the $\Q$-factorial and toric variety determined by the fan $\Si'$. Clearly $Z$ is complete if and only if the support $\langle V\rangle=|\Si'|$ coincides with $N_\R\cong\R^n$. By construction the total coordinate space $\overline{Z}$ of $Z$ is given by
\begin{equation*}
  \overline{Z}=\Spec(\Cox(Z))=\K^m=\Spec(\Cox(W))=\overline{W}\,.
\end{equation*}
Moreover, calling $\B'$ the bunch associated by Gale duality with $\Si'$, the inclusion $\Si\subseteq\Si'$ of simplicial fans, translates into the inclusion $\B\subseteq\B'$ of bunches. Then
\begin{eqnarray*}
  \I rr(Z)=&\left(\prod_{i\not\in I}x_i\,|\,I\in\I_{\Si'}\right)=\left(\prod_{j\in J}x_j\,|\,\langle Q_J\rangle\in\B'\right)& \\
  \supseteq&\left(\prod_{j\in J}x_j\,|\,\langle Q_J\rangle\in\B\right)=\left(\prod_{i\not\in I}x_i\,|\,I\in\I_\Si\right)&=\I rr(W)=\widetilde{\I rr}\\
  &\Rightarrow B_Z\subseteq B_W=\widetilde{B}&
\end{eqnarray*}
where the latter is an inclusion of Zariski closed subset of $\overline{Z}=\overline{W}$. Setting $\widehat{Z}:=\overline{Z}\setminus B_Z$ we get naturally the open embeddings
\begin{eqnarray*}
  j_Z&:&\xymatrix{\widehat{Z}\hookrightarrow \overline{Z}}\\
  \widehat{\iota}&:&\xymatrix{\widehat{W}:=\overline{W}\setminus B_W\ar@{^{(}->}[r] & \overline{Z}\setminus B_Z=:\widehat{Z}}
\end{eqnarray*}
which are clearly equivariant \wrt the $G$-action on $\overline{W}=\overline{Z}$. Such an action gives rise to the characteristic space $p_Z:\widehat{Z}\twoheadrightarrow Z$ and, so, to a 1-free geometric quotient by Cox Theorem \cite[Thm.~2.1]{Cox}, as $Z$ is $\Q$-factorial and without torus factors, being a sharp partial completion of $W$. In particular, $j_Z(\widehat{Z})=\overline{Z}^s=\overline{Z}^{ss}$.  Passing to the quotient by the equivariant action of $G$, the open embedding $\widehat{\iota}$ gives rise to an open embedding $\iota:W\hookrightarrow Z$. By construction, the closure $X':=\overline{\iota\circ i(X)}$, inside $Z$, is a sharp partial completion of $X$ whose natural embedding $j:X'\hookrightarrow Z$ is a neat closed embedding. Sharpness conditions ensure that $\rk(\Cl(X'))=\rk(\Cl(X)$, hence finite. Finally, by construction $\Cox(X')\cong\Cox(X)\cong\Cox(W)/I$, being the relations defining $I$ unchanged. Then $X'$ is a wMDS. If $Z$ is complete then also $X'$ is complete.

(2): The given diagram follows immediately by construction explained in the previous part (1) and Proposition~\ref{prop:neat}. The only point deserving some word of explanation is the surjectivity of $\iota^*:\Pic(Z)\to\Pic(W)$, which is immediately obtained by observing that every local equation of a Cartier divisor $D$ on $Z$ is actually a local equation of the Cartier divisor $\iota^*D$ on $W$, since $\iota$ is an open embedding.

(3): The given diagram follows immediately by construction explained in the previous part (1) and Proposition~\ref{prop:Qfactorial}. In particular notice that $Z$ has been chosen by asking that $\g=\Nef(Z)$ is a g-cell inside $\Nef(W)$.
\end{proof}

\begin{remark}
  Recalling \S~\ref{ssez:Irr&bunches} and in particular relation (\ref{K-shrink}), up to shrink $K$, pullback morphisms $i^\#,i^*,\iota^\#,\iota^*$ allows us to extend diagram (\ref{div-diagram-embedding}) as follows
  \begin{equation}\label{div-diagram-embedding-2}
      \xymatrix{0 \ar[r]&  M \ar[r]^-{div_Z}_-{V^T}\ar[d]_-{\mathbf{I}_{n}} &
\Div_\T(Z) \ar[r]^-{d_Z}_-{Q\oplus\Ga}\ar[d]^-{\iota^\#}_-{\mathbf{I}_{m}} & \Cl(Z) \ar[d]^-{\iota^*}_-{\mathbf{I}_{r}\oplus\overline{\mathbf{I}}_s} \ar[r]&0\\
      0 \ar[r]&  M \ar[r]^-{div_W}_-{V^T}\ar[d]_-{\mathbf{I}_{n}} &
\Div_\T(W) \ar[r]^-{d_W}_-{Q\oplus\Ga}\ar[d]^-{i^\#}_-{\mathbf{I}_{m}} & \Cl(W) \ar[d]^-{i^*}_-{\mathbf{I}_{r}\oplus\overline{\mathbf{I}}_s} \ar[r]&0\\
0 \ar[r]& K_0 \ar[r]_-{V^T}&K\ar[r]^-{d_{K}}_-{Q\oplus\Ga} & \Cl (X) \ar[r]& 0 \\
 }
\end{equation}
where representative matrices are chosen \wrt a fixed basis of the character group $M=\Hom(\T,\Z)$ and the standard basis of torus invariant divisors.
\end{remark}

\begin{proposition}\label{prop:cellinNef}
  Let $X$ be a complete wMDS with Cox basis $\X$ and $i:X\hookrightarrow W$ its $\X$-canonical toric embedding. Then $X$ is a MDS if and only if $\Nef(X)\cong\Nef(W)$ is a finite union of chambers.
\end{proposition}

\begin{proof}
   Assume that $X$ is a MDS. Then \cite[Prop.~2.11\,(3)]{Hu-Keel} implies that $\Nef(X)$ is a finite union of chambers, each of them associated with a projective completion of the canonical ambient toric variety $W$.
  Conversely if $\g\subseteq\Nef(X)\cong\Nef(W)$ is a g-chamber, then $X$ is a fillable wMDS and the previous Theorem~\ref{thm:complete-emb} exhibits a completion $Z(\g,\Si')$ of $W$ which is projective by Propositions~\ref{prop:quasiproiettivo} and \ref{prop:cell=camera}. Then $X$ is a projective wMDS, hence a MDS.
\end{proof}

An obvious consequence of the previous Propositions~\ref{prop:camera-filling} and \ref{prop:cellinNef} is the following

\begin{corollary}\label{cor:MDS=>polarizzato}
  Every MDS is a complete fillable wMDS.
\end{corollary}

\begin{remark}\label{rem:polar_noMDS}
Notice that a converse of the previous Corollary~\ref{cor:MDS=>polarizzato} does not hold in general, as many well known examples of complete, $\Q$-factorial and non-projective toric varieties show. Consider e.g. either the Oda's example of a smooth, complete and non-projective toric threefold of Picard number 4, given in his famous book \cite[p.~84]{Oda}, or $\Q$-factorial, complete and non-projective 3-dimensional toric varieties, of Picard number 3, given by fans $\Si,\Si'$ in the previous Example~\ref{ex:cells}: both of them admit $\Nef$ cone given by a g-cell, meaning that they are fillable wMDS. But they cannot be MDS being non-projective.
\end{remark}

\begin{remark}\label{rem:fillableHp}
  A natural question is if the strong hypothesis, for a wMDS in Theorem~\ref{thm:complete-emb}, of being fillable could be dropped. In fact it seems to be an extra-condition when the statement of this theorem is compared with that of \cite[Prop.~2.11]{Hu-Keel}. Actually the previous Corollary~\ref{cor:MDS=>polarizzato} clarifies that the existence of a filling chamber is guaranteed when we are dealing with Mori dream spaces. Anyway one could ask if also being a wMDS may imply the existence of a filling cell in $\Nef(X)$.

  \noindent The answer is negative, in general, as the following Example~\ref{ex:noncompletabile} shows. Nevertheless, this example does not exclude the possibility that being a \emph{complete} wMDS may imply the existence of a filling cell, as explained in the following Remark~\ref{rem:exnoncompleto}. Unfortunately I was not able to find a counterexample to such a possibility: the point is exhibiting a complete wMDS $X$ whose canonical ambient toric variety $W$ is a non-complete and non sharply completable one! Does it exist?
\end{remark}

\subsection{Discussing hypotheses of Theorem \ref{thm:complete-emb}: some examples}\label{ssez:esempi}
In the present section we give two instances of possible occurrences in the completion procedure of Theorem~\ref{thm:complete-emb}. Namely, in the first example \ref{ex:ADHL-Ex3.2.5.8}, we will discuss a case in which $X$ is complete and fillable, hence $X\cong X'$ admitting a sharp completion $Z$ of the canonical ambient toric variety $W$. On the contrary, in the second Example~\ref{ex:noncompletabile} we will consider the case of a non-complete wMDS which is not fillable.

\begin{example}\label{ex:ADHL-Ex3.2.5.8} Let us start by considering a case in which $X$ is complete and coincides with the induced sharp completion $X'$, as described in item (1) of  Theorem~\ref{thm:complete-emb}. At this purpose consider the MDS $X$ given in \cite[Ex.~3.2.5.8]{ADHL}. This is also considered, up to a isomorphism, in the \emph{Cox ring database} \cite{CRdb}, where it is reported as the id no.~67.

\begin{figure}
\begin{center}
\includegraphics[width=8truecm]{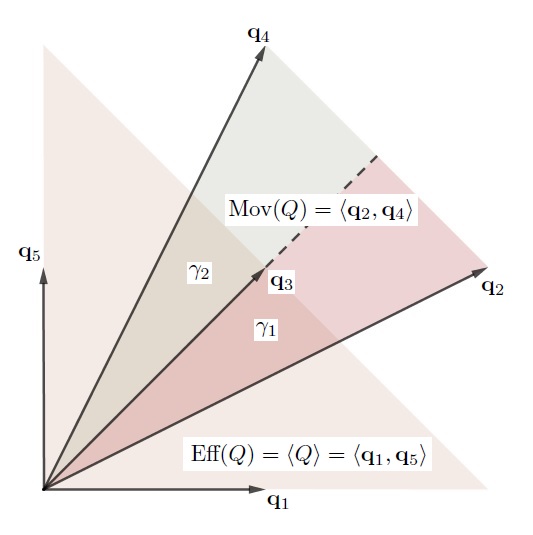}
\caption{\label{Fig2}The GKZ decomposition of the canonical ambient toric variety $W$ in Ex.~\ref{ex:ADHL-Ex3.2.5.8}.}
\end{center}
\end{figure}

\noindent Consider the grading map $d_K:K=\Z^5\twoheadrightarrow\Z^2$ represented by the weight matrix
\begin{equation*}
  Q=\left(
      \begin{array}{ccccc}
        1 & 2 & 1 & 1 & 0 \\
        0 & 1 & 1 & 2 & 1 \\
      \end{array}
    \right) =\left(
             \begin{array}{ccc}
               \q_1 & \cdots & \q_5 \\
             \end{array}
           \right)
\end{equation*}
and the quotient algebra $R=\K[x_1,\ldots,x_5]/(x_1x_4+x_2x_5+x_3^2)$, graded by $d_K$. This is consistent since the relation defining $R$ is homogeneous \wrt such a grading. Moreover $R$ turns out to be a Cox ring with $\X:=\{\overline{x}_1,\ldots,\overline{x}_5\}$ giving a Cox basis of $R$. Then $\overline{X}:=\Spec(R)\subseteq \Spec\K[\x]=:\overline{W}$ defines the total coordinate space of a wMDS $X:=\widehat{X}/(\K^*)^2$ where
\begin{equation*}
  \widehat{X}=\overline{X}\backslash B_X\ \text{being}\ B_X=\mathcal{V}(\irr(X))\ \text{and}\ \irr(X)=(\overline{x}_1\overline{x}_5,\overline{x}_2\overline{x}_4,\overline{x}_1\overline{x}_3\overline{x}_4,
  \overline{x}_2\overline{x}_3\overline{x}_5)
\end{equation*}
The canonical ambient toric variety of $X$ is the given by $W:=\widehat{W}/(\K^*)^2$ where
\begin{equation*}
  \widehat{W}=\overline{W}\backslash \widetilde{B}\ \text{being}\ \widetilde{B}=\mathcal{V}(\widetilde{\irr})\ \text{and}\ \widetilde{\irr}=(x_1x_5,x_2x_4,x_1x_3x_4,
  x_2x_3x_5)
\end{equation*}
Notice that $W=W(\Si)$ is the toric variety given by a fan matrix $V$, gale dual to $Q$, e.g.
\begin{equation*}
  V:=\left(
     \begin{array}{ccccc}
       1 & 0 & 0 & -1 & 2 \\
       0 & 1 & 0 & -2 & 3 \\
       0 & 0 & 1 & -1 & 1 \\
     \end{array}
   \right)=\left(
             \begin{array}{ccc}
               \v_1 & \cdots & \v_5 \\
             \end{array}
           \right)
\end{equation*}
and fan $\Si$ generated by the following collection of maximal cones
\begin{equation*}
  \Si(\max):=\{\langle\v_2,\v_3,\v_4\rangle\,\langle\v_1,\v_3,\v_5\rangle,\langle\v_2,\v_5\rangle,
  \langle\v_1,\v_4\rangle\}\,.
\end{equation*}
The GKZ decomposition of $\overline{\Eff}(X)\cong\overline{\Eff}(W)=\langle Q\rangle=:\Eff(Q)$ is represented in Fig.~\ref{Fig2}. In particular it turns out that
$$\Nef(X)=\Mov(X)\cong\Mov(W)=\Nef(W)=\Mov(Q)=\langle\q_2,\q_4\rangle$$
Following notation introduced in Theorem~\ref{thm:complete-emb}, it is the union of two g-chambers $\g_1=\langle\q_2,\q_3\rangle$ and $\g_2=\langle\q_3,\q_4\rangle$, giving the Nef cones of the only two possible sharp completions $Z_1(\g_1,\Si_1),Z_2(\g_2,\Si_2)$ of $W$. In particular one has
\begin{equation*}
  \SF(V)=\P\SF(V)=\{\Si_1,\Si_2\}
\end{equation*}
where $\Si_1$ and $\Si_2$ are simplicial and complete fans generated by the following collection of maximal cones
\begin{eqnarray*}
  \Si_1(3)&=&\{\langle\v_1,\v_3,\v_4\rangle\,\langle\v_1,\v_3,\v_5\rangle\,\langle\v_1,\v_4,\v_5\rangle\,
  \langle\v_2,\v_3,\v_4\rangle\,\langle\v_2,\v_3,\v_5\rangle\,\langle\v_2,\v_4,\v_5\rangle\}\\
  \Si_2(3)&=&\{\langle\v_1,\v_2,\v_4\rangle\,\langle\v_1,\v_2,\v_5\rangle\,\langle\v_1,\v_3,\v_4\rangle\,
  \langle\v_1,\v_3,\v_5\rangle\, \langle\v_2,\v_3,\v_4\rangle\,\langle\v_2,\v_3,\v_5\rangle\}
\end{eqnarray*}
For $k=1,2$, the associated (projective) completion $Z_k$ of $W$ encodes the completion $X'_k:=\widehat{X}'_k/(\K^*)^2$ of $X$, where
\begin{equation*}
  \widehat{X}'_k=\overline{X}\backslash B_k\ \text{being}\ B_k=\mathcal{V}(\irr(X'_k))
  \end{equation*}
  and
  \begin{eqnarray*}
    \irr(X'_1) &=& (\overline{x}_1\overline{x}_3,\overline{x}_1\overline{x}_4,\overline{x}_1\overline{x}_5,
    \overline{x}_2\overline{x}_3,\overline{x}_2\overline{x}_4,\overline{x}_2\overline{x}_5)\\
    \irr(X'_2) &=& (\overline{x}_1\overline{x}_4,\overline{x}_1\overline{x}_5,\overline{x}_2\overline{x}_4,
    \overline{x}_2\overline{x}_5,\overline{x}_3\overline{x}_4,\overline{x}_3\overline{x}_5)
  \end{eqnarray*}
Notice that $X$ is actually complete (then a MDS) since $X= X'_k$, for both $k=1,2$. Therefore $j_k=\iota_k\circ i:X\hookrightarrow Z_k$ is a neat embedding, for both $k=1,2$.
\end{example}

\begin{example}\label{ex:noncompletabile} The following is an example of a wMDS whose canonical toric am\-bient variety is a non sharply completable one. This fact shows that the hypothesis of existence of a filling cell in $\Nef(X)$, as given in Theorem~\ref{thm:complete-emb}, cannot be dropped.

\noindent Consider the grading map $d_K:K=\Z^7\twoheadrightarrow\Z^3$ represented by the weight matrix
\begin{equation*}
  Q=\left(
      \begin{array}{ccccccc}
        1 & 1 & 0 & 1 & 0 & 1 & 0 \\
        0 & 1 & 1 & 1 & 1 & 0 & 0 \\
        0 & 0 & 0 & 1 & 1 & 1 & 1 \\
      \end{array}
    \right) =\left(
             \begin{array}{ccc}
               \q_1 & \cdots & \q_7 \\
             \end{array}
           \right)
\end{equation*}
and the quotient algebra
$$
R=\K[x_1,\ldots,x_7]/(x_1x_3x_7+x_1x_5+x_2x_7+ x_3x_6+x_4)
$$
graded by $d_K$.
Moreover, the divisor associated with the generator $\overline{x}_i$, given by cutting with $\{x_i=0\}$, turns out to be a prime divisor, so giving that $\overline{x}_i$ is $\Cl(X)$-prime \cite[Prop.\,1.5.3.5\,(iii)]{ADHL}. Then $R$ turns out to be a Cox ring with $\X:=\{\overline{x}_1,\ldots,\overline{x}_7\}$ giving a Cox basis of $R$,  $\overline{X}:=\Spec(R)\subseteq \Spec\K[\x]=:\overline{W}$ defines the total coordinate space of a wMDS $X:=\widehat{X}/(\K^*)^3$ where $\widehat{X}=\overline{X}\backslash B_X$, being $B_X=\mathcal{V}(\irr(X))$ and
\begin{equation*} \irr(X)=\left(
                            \begin{array}{c}
                              \overline{x}_2\overline{x}_5\overline{x}_6,\overline{x}_1\overline{x}_2\overline{x}_3\overline{x}_7,
\overline{x}_1\overline{x}_2\overline{x}_5\overline{x}_7,\overline{x}_1\overline{x}_3\overline{x}_4\overline{x}_6,
\overline{x}_1\overline{x}_3\overline{x}_4\overline{x}_7, \\
\overline{x}_1\overline{x}_3\overline{x}_5\overline{x}_6,\overline{x}_1\overline{x}_3\overline{x}_5\overline{x}_7,
  \overline{x}_1\overline{x}_3\overline{x}_6\overline{x}_7, \overline{x}_1\overline{x}_4\overline{x}_5\overline{x}_6,
  \overline{x}_1\overline{x}_4\overline{x}_5\overline{x}_7,\\
  \overline{x}_2\overline{x}_3\overline{x}_4\overline{x}_6,
  \overline{x}_2\overline{x}_3\overline{x}_4\overline{x}_7,\overline{x}_2\overline{x}_3\overline{x}_6\overline{x}_7, \overline{x}_2\overline{x}_4\overline{x}_5\overline{x}_7\\
                            \end{array}
                          \right)
\end{equation*}
The canonical ambient toric variety of $X$ is then given by $W:=\widehat{W}/(\K^*)^3$ where $\widehat{W}=\overline{W}\backslash \widetilde{B}$, being $\widetilde{B}=\mathcal{V}(\widetilde{\irr})$ and $\widetilde{\irr}=\pi_\X^{-1}(\irr(X))$. In particular the bunch $\B$ of $W$ is generated by the following collection of minimal cones
\begin{equation*}
  \B(\min):=\left\{
               \begin{array}{c}
                 \langle\q_2,\q_5,\q_6\rangle,\langle\q_1,\q_2,\q_3,\q_7\rangle,
\langle\q_1,\q_2,\q_5,\q_7\rangle,\langle\q_1,\q_3,\q_4,\q_6\rangle,
\\
\langle\q_1,\q_3,\q_4,\q_7\rangle,\langle\q_1,\q_3,\q_5,\q_6\rangle,\langle\q_1,\q_3,\q_5,\q_7\rangle,\langle
  \q_1,\q_3,\q_6,\q_7\rangle,\\
  \langle \q_1,\q_4,\q_5,\q_6\rangle,\langle
  \q_1,\q_4,\q_5,\q_7\rangle,\langle\q_2,\q_3,\q_4,\q_6\rangle,\\
  \langle
  \q_2,\q_3,\q_4,\q_7\rangle,\langle\q_2,\q_3,\q_6,\q_7\rangle,\langle \q_2,\q_4,\q_5,\q_7\rangle\\
               \end{array}
             \right\}
\end{equation*}
Then $W$ is a $\Q$-factorial non-complete toric variety.
\begin{figure}
\begin{center}
\includegraphics[width=8truecm]{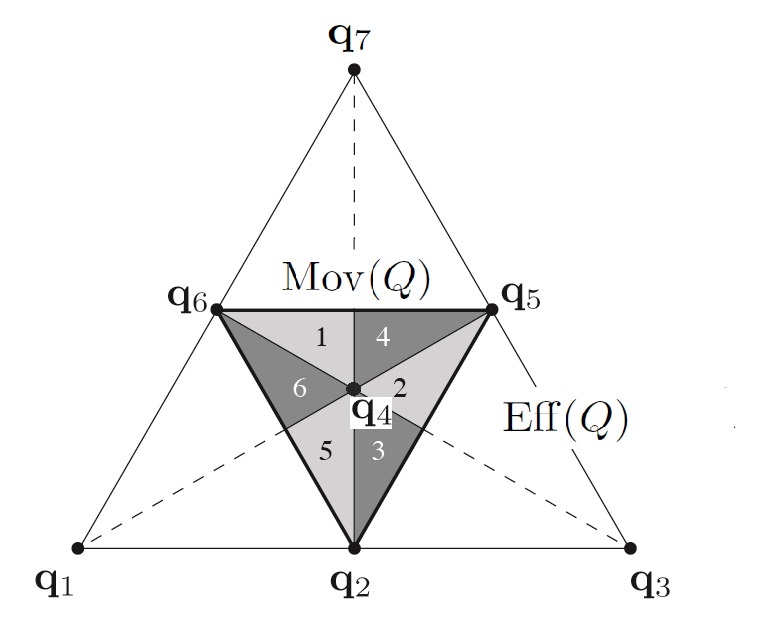}
\caption{\label{Fig3}The GKZ decomposition of the canonical ambient toric variety $W$ in Ex.~\ref{ex:noncompletabile}.}
\end{center}
\end{figure}
The GKZ decomposition of $\overline{\Eff}(X)\cong\overline{\Eff}(W)=\Eff(Q)$ is represented in Fig.~\ref{Fig3}, where chambers from 1 to 6 are the Nef cones associated with elements of $\P\SF(V)=\SF(V)$, being $V$ a Gale dual matrix of $Q$, given e.g. by
\begin{equation}\label{V}
  V:=\left(
     \begin{array}{ccccccc}
       1 & 0 & 0 & 0 & 0 & -1 & 1 \\
       0 & 1 & 0 & 0 & -1 & -1 & 2\\
       0 & 0 & 1 & 0 & -1 & 0 & 1 \\
       0 & 0 & 0 & 1 & -1 & -1 & 1 \\
     \end{array}
   \right)=\left(
             \begin{array}{ccc}
               \v_1 & \cdots & \v_7 \\
             \end{array}
           \right)
\end{equation}
In particular
\begin{equation*}
  \Nef(X)\cong\Nef(W)=\bigcap_{\tau\in\B}\tau = \langle\q_4\rangle
\end{equation*}
which is strictly contained in every chamber from 1 to 6. Consequently, none of these chambers can be a filling cell in $\Nef(X)$ and $W$ cannot admit any sharp completion.
\end{example}

\begin{remark}\label{rem:exnoncompleto}
An interesting question is if \emph{assuming $X$ complete in Theorem~\ref{thm:complete-emb} may guarantee the existence of a filling cell in $\Nef(X)$}.

\noindent Notice that the previous Example~\ref{ex:noncompletabile} cannot exclude such a possibility, as $X$ is not complete, therein.  In fact the six fans in $\SF(V)$ are generated by the six collections of maximal cones
$$\forall\,i=1,\ldots,6\quad\Si_i(4)=\{\langle V_I\rangle\,|\,I\in\I_{\Si_i}(4)\}$$
where
\begin{eqnarray*}
  I_{\Si_1}(4) &=& \left\{\{3, 4, 5, 6\}, \{2, 4, 5, 6\}, \{2, 3, 5, 6\}, \{2, 3, 4, 6\}, \{1, 3, 4, 5\}, \{1, 2, 4, 5\},\right.\\
   &&\left.\{1, 2, 3, 5\}, \{2, 3, 4, 7\}, \{1, 3, 4, 7\}, \{1, 2, 4, 7\}, \{1, 2, 3, 7\}\right\} \\
  I_{\Si_2}(4) &=& \left\{\{2, 4, 5, 6\}, \{1, 4, 5, 6\}, \{1, 2, 5, 6\}, \{1, 3, 4, 6\}, \{3, 4, 6, 7\}, \{2, 4, 6, 7\}, \right.\\
  && \left.\{1, 3, 6, 7\}, \{1, 2, 6, 7\}, \{1, 2, 4, 5\}, \{1, 3, 4, 7\}, \{1, 2, 4, 7\}\right\} \\
  I_{\Si_3}(4) &=&\left\{\{2, 4, 5, 6\}, \{1, 4, 5, 6\}, \{2, 5, 6, 7\}, \{1, 3, 4, 6\}, \{3, 4, 6, 7\}, \{2, 4, 6, 7\}, \right.\\
  && \left. \{1, 5, 6, 7\}, \{1, 3, 6, 7\}, \{2, 4, 5, 7\}, \{1, 4, 5, 7\}, \{1, 3, 4, 7\}\right\} \\
  I_{\Si_4}(4) &=& \left\{\{2, 4, 5, 6\}, \{2, 3, 4, 6\}, \{1, 4, 5, 6\}, \{1, 2, 5, 6\}, \{1, 3, 4, 6\}, \{1, 2, 3, 6\}, \right.\\
  && \left.\{1, 2, 4, 5\}, \{2, 3, 4, 7\}, \{1, 3, 4, 7\}, \{1, 2, 4, 7\}, \{1, 2, 3, 7\}\right\} \\
  I_{\Si_5}(4) &=& \left\{\{3, 4, 5, 6\}, \{2, 4, 5, 6\}, \{3, 5, 6, 7\}, \{2, 5, 6, 7\}, \{3, 4, 6, 7\}, \{2, 4, 6, 7\}, \right.\\
  && \left.\{1, 3, 4, 5\}, \{2, 4, 5, 7\}, \{1, 4, 5, 7\}, \{1, 3, 5, 7\}, \{1, 3, 4, 7\}\right\} \\
  I_{\Si_6}(4) &=& \left\{\{3, 4, 5, 6\}, \{2, 4, 5, 6\}, \{2, 3, 5, 6\}, \{2, 3, 4, 6\}, \{1, 3, 4, 5\}, \{2, 4, 5, 7\}, \right.\\
  && \left.\{2, 3, 5, 7\}, \{2, 3, 4, 7\}, \{1, 4, 5, 7\}, \{1, 3, 5, 7\}, \{1, 3, 4, 7\}\right\}
\end{eqnarray*}
On the other hand, by Gale duality the collection of minimal cones $\B(\min)$ determines the collection of maximal cones of the fan $\Si$ of $W$, namely
$$\Si(\max)=\{\langle V_I\rangle\,|\,I\in\I_\Si(\max)\}$$
with
\begin{eqnarray*}
  \I_\Si(\max) &=& \left\{\{1, 3, 4, 7\}, \{4, 5, 6\}, \{3, 4, 6\}, \{2, 5, 7\}, \{2, 5, 6\}, \{2, 4, 7\}, \{2, 4, 6\}, \right.\\
  && \left. \{2, 4, 5\}, \{2, 3, 7\}, \{2, 3, 6\}, \{1, 5, 7\}, \{1, 5, 6\}, \{1, 4, 5\}, \{1, 3, 6\}\right\}
\end{eqnarray*}
Consider the following 3-cones not belonging to $\Si$
\begin{eqnarray*}
  \langle\v_1,\v_2,\v_5\rangle&\in&\Si_1\cap\Si_2\cap\Si_4\\
  \langle\v_2,\v_6,\v_7\rangle&\in&\Si_2\cap\Si_3\cap\Si_5\\
  \langle\v_3,\v_5,\v_6\rangle&\in&\Si_1\cap\Si_5\cap\Si_6
\end{eqnarray*}
For $k=1,\ldots,6$, let $X'_k$ be the closure, inside the sharp completion $Z_k(\Si_k)$ of $W$, of the image $\iota_k\circ i(X)$, being $\iota_k:W\hookrightarrow Z_k$ the natural open embedding. The closures of the toric orbits of each of the 3-cones above give rise to three 1-cycles intersecting $X'_k$, for the associated value of $k$, in a finite number of points living outside of the irrelevant loci.   Namely, one has
\begin{eqnarray*}
\overline{W}\supseteq\mathcal{V}(x_1,x_2,x_5,x_3x_6+x_4)& \Rightarrow & [0:0:1:1:0:-1:1]\in X'_1\cap X'_2\cap X'_4\\
\overline{W}\supseteq\mathcal{V}(x_2,x_6,x_7,x_1x_5+x_4)& \Rightarrow & [1:0:1:1:-1:0:0]\in X'_2\cap X'_3\cap X'_5\\
\overline{W}\supseteq\mathcal{V}(x_3,x_5,x_6,x_2x_7+x_4)& \Rightarrow & [1:1:0:1:0:0:-1]\in X'_1\cap X'_5\cap X'_6
\end{eqnarray*}
Clearly points exhibited on the right side cannot belong to $X$, showing that $X$ is not complete.
\end{remark}

\subsection{Projective embedding of lower rank complete wMDS}\label{ssez:proj-emb}

In the recent paper \cite{RT-r2proj}, jointly written with L.~Terracini, we proved\footnote{In \cite{RT-r2proj} we assumed $\C$ as ground field. Actually all techniques and arguments therein employed are completely extendable to a general field $\K=\overline{\K}$ with $\Char\K=0$.} the following

\begin{theorem}[Thm.~3.2 in \cite{RT-r2proj}]\label{thm:r2}
  Every $\Q$-factorial complete toric variety of Picard number $r\leq 2$ is projective.
\end{theorem}

First of all let us observe that this result allows us to prove a converse of Corollary~\ref{cor:MDS=>polarizzato} for a complete fillable wMDS of Picard number $r\leq 2$. In fact Theorem~\ref{thm:complete-emb} ensures that a complete fillable wMDS $X$ can be embedded in a complete $\Q$-factorial variety $Z$ ha\-ving the same Picard number of $X$. Hence Theorem~\ref{thm:r2} shows that $Z$ is projective, so giving a projective embedding of $X$, too. This gives that

\begin{corollary}
  Every complete wMDS of Picard number $r\leq 2$ is fillable if and only if it is a MDS.
\end{corollary}

Consequently, examples given in Example~\ref{ex:cells} and recalled in Remark~\ref{rem:polar_noMDS}, turn out to be sharp counterexamples \wrt the Picard number.

Now we are going to show that

\begin{theorem}\label{thm:r<3fillable}
  Every complete wMDS of Picard number $r\leq 2$ is fillable.
\end{theorem}

\begin{proof} We will actually prove a stronger statement:
\begin{itemize}
\item[(*)] \emph{if $X$ is a complete wMDS of Picard number $r\leq2$ then $\Nef(X)$ contains a full-dimensional GKZ-cone}.
\end{itemize}
Then the thesis follows by recalling Propositions~\ref{prop:camera} and \ref{prop:camera-filling}.

To prove (*), notice that Remark\,\ref{rem:corrispondenze}\,(b) shows that
$$\Nef(X)\cong\Nef(W)=\bigcap_{I\in\I_\Si}\langle Q^I\rangle$$
being $\Si$ a fan of the canonical ambient toric variety $W$. Corollary~\ref{cor:WèQ-f} shows that $W$ is $\Q$-factorial, that is $\Si$ is simplicial. Then $\langle Q^I\rangle$ is a 2-dimensional cone, for every $I\in\I_\Si$. Moreover Corollary~\ref{cor:Qpositive} shows that the weight matrix $Q$ is a $W$-matrix and a fan matrix $V$ of $W$ is an $F$-matrix. This is enough to ensure that lemmas~1.2 and 2.1 in \cite{RT-r2proj} still hold for the toric variety $W$. Then the proof of (*) goes on exactly as the proof of \cite[Thm.~2.2]{RT-r2proj}, giving that the 2-dimensional GKZ-cone contained in $\Nef(W)$ is the minimal 2-cone contained in every cone of the bunch of cones $\B=\{\langle Q^I\rangle\,|\,I\in\I_\Si\}$.
\end{proof}

Recalling Proposition~\ref{prop:cellinNef} and Theorem~\ref{thm:complete-emb}, the previous result gives im\-me\-dia\-te\-ly the following

\begin{corollary}\label{cor:r<3}
  Every complete wMDS of Picard number $r\leq 2$ is a MDS.
\end{corollary}

Let us conclude the present section by recalling that every smooth and complete toric variety of Picard number $r\leq 3$ is projective, by a well known result of Kleinschmidt and Sturmfels \cite{Kleinschmidt-Sturmfels}. This clearly gives the following

\begin{corollary}\label{cor:KSforwMDS}
Every smoothly fillable and complete wMDS of Picard number $r\leq 3$ is a MDS.
\end{corollary}

Notice that the Oda's example recalled in Remark~\ref{rem:polar_noMDS} is then a sharp counterexample \wrt an extension of Corollary~\ref{cor:KSforwMDS} for higher values of the Picard number $r$: in fact it is a smoothly fillable wMDS with $r=4$ and non-projective.

\section{Birational geometry of complete weak Mori Dream Spaces}\label{sez:birazionale}

In the present section, a wMDS is considered by the birational point of view, aiming to a possible extension, beyond the projective setup, of Hu-Keel results about the termination of a MMP for every divisor and the classification of rational contractions \cite[Prop.~1.11,~2.11]{Hu-Keel}. As a direct consequence of \cite[Thm.~4.3.3.1]{ADHL}, here recalled by the following Lemma~\ref{lem:sQm}, these results extend quite naturally to a complete wMDS.

\begin{definition}[Small $\Q$-factorial modification]\label{def:sQm}
  A birational map $f:X\dashrightarrow Y$, between irreducible, complete and $\Q$-factorial algebraic varieties, is called a \emph{small $\Q$-factorial modification} (s$\Q$m) if it is biregular in codimension 1 i.e. there exist Zariski open subsets $U\subseteq X$ and $V\subseteq Y$ such that $f|_U:U\stackrel{\cong}{\rightarrow} V$ is biregular and $\codim (X\setminus U)\geq 2$\,, $\codim(Y\setminus V)\geq 2$\,.
\end{definition}

\begin{lemma}[\cite{ADHL}, Thm.~4.3.3.1]\label{lem:sQm}
A $\Q$-factorial and complete algebraic variety $X$ is a wMDS if and only if there exists a \sqm $f:X\dashrightarrow X'$ such that $X'$ is a MDS.
\end{lemma}

\begin{proof}[Sketch of proof] The necessary condition is trivial. In fact, let  $f:X\dashrightarrow X'$ be a \sqm with $X'$ given by a MDS.
  Then $f$ induces isomorphisms $\Cl(X)\cong \Cl(X')$ and $\Cox(X)\cong\Cox(X')$, as $f$ is an isomorphism in codimension 1, so giving that $X$ has to be a wMDS.

The sufficient condition is precisely \cite[Thm.~4.3.3.1]{ADHL}: here just a rough idea, of how constructing the \sqm $f$, is given,  referring the interested reader to \cite[\S~4.3.3]{ADHL} and \cite[\S~15.4]{CLS}, for any further detail.

First of all recall that, given the $\X$-canonical toric embedding $i:X\hookrightarrow W$, the cone $\Mov(X)\cong\Mov(W)$ is a full dimensional convex cone inside $\Eff(X)\cong\Eff(W)$ and the support of the secondary fan $\Ga(Q)|_{\Mov(W)}$, being $Q$ a weight matrix of $W$. Then, there certainly exists  a full dimensional GKZ-cone $\g'$ such that $\g:=\Nef(X)\preceq\g'$. Since $X$ is complete, hypotheses of Prop.~\ref{prop:camera} are satisfied and $\g'$ turns out to be a g-chamber. Calling $V$ a Gale dual matrix of $Q$, let $\Si'\in\P\SF(V)$ be the fan associated with $\g'$ and $Z=Z(\Si')$ be the related $\Q$-factorial and projective toric variety. Then, recalling notation introduced in Rem.~\ref{rems:wMDS}~(4),
\begin{equation*}
    \overline{Z}\cong\Spec\K[\X]\cong\overline{W}
\end{equation*}
and $\Cox(X)\cong\K[\X]/I$. In particular, $X=\widehat{X}/G$\,, where $G=\Hom(\Cl(X),\K^*)$ is the characteristic quasi-torus of $X$, and by setting
\begin{eqnarray*}
\irr(X')&:=&\pi_\X(\irr(Z))=\irr(Z)/\irr(Z)\cap I\\
B_{X'}&:=&\mathcal{V}(\irr(X'))\\
\widehat{X}'&:=&\overline{X}\setminus B_{X'}\\
X'&:=&\widehat{X'}/G
\end{eqnarray*}
one gets the definition of $f$ by the following diagram
\begin{equation*}
    \xymatrix{&\widehat{X}\cap \widehat{X}'\ar[dl]_-{p_{X}|_{\widehat{X}\cap \widehat{X}'}}\ar[dr]^-{p_{X'}|_{\widehat{X}\cap \widehat{X}'}}\\
                X\ar@{-->}^f[rr]&&X'}
\end{equation*}
 that is $f(x):=p_{X'}(p_{X}^{-1}(x))$, for every $x\in p_{X}(\widehat{X}\cap \widehat{X}')$. It remains to check that:
 \begin{itemize}
     \item $X'$ is a well defined MDS,
     \item $p_{X}(\widehat{X}\cap \widehat{X}')$ and $ p_{X'}(\widehat{X}\cap \widehat{X}')$ are open Zariski subsets of $X$ and $X'$, respectively, whose complementary sets have both codimension at least 2\,,
     \item $f:p_{X}(\widehat{X}\cap \widehat{X}')\cong p_{X'}(\widehat{X}\cap \widehat{X}')$\,, so giving an isomorphism in codimension 1 between $X$ and $X'$\,.
 \end{itemize}
For all the necessary details, the interested reader is referred to the extensive treatments cited above.
\end{proof}

\subsection{The Minimal Model Program (MMP)}\label{ssez:MMP}

Mori Dream Spaces have been introduced to give a class of varieties on which a MMP can be carried out for any divisor \cite[Prop.~1.11\,(1)]{Hu-Keel}. By Lemma~\ref{lem:sQm}, this fact extends immediately to complete weak Mori Dream Spaces.

Let us recall some standard notation.

\begin{definition}[MMP for a divisor]\label{def:MMP}
  Let $X$ be a $\Q$-factorial and complete algebraic variety and consider $D\in\Div(X)$. A MMP for $D$, or $D$-MMP, is a finite sequence of birational transformations
  \begin{equation*}
    \xymatrix{X=:X_0\ar@{-->}[r]^-{f_1}&X_1\ar@{-->}[r]^-{f_2}&X_2\ar@{-->}[r]^-{f_3}&
    \cdots\ar@{-->}[r]^-{f_{l-1}}&X_{l-1}\ar@{-->}[r]^-{f_l}&X_l=:X^*}
  \end{equation*}
  such that:
  \begin{enumerate}
    \item $X_i$ is a $\Q$-factorial and complete algebraic variety,  for every $0<i\leq l$;
    \item setting $D_0=D$ and $D_i:=(f_i)_*D_{i-1}:=\overline{f_i(D_i)}\in\Div(X_i)$ be the birational transform of $D_{i-1}\in\Div(X_{i-1})$, the birational map $f_i$ is a $D_{i-1}$-negative contraction (i.e. $D_{i-1}\cdot C<0$ for every complete and irreducible curve $C\subset X_{i-1}$ such that $f_{i-1}(C)$ is a point); in particular $f_i$ is either a divisorial contraction or a flip of a $D_{i-1}$-negative small contraction, for every $0<i\leq l$;
    \item either $D^*:=D_l$ is nef and $X^*:=X_l$ is called a \emph{$D$-minimal model of $X$}, or $X^*$ admits a $D^*$-negative contraction over a lower dimensional variety, giving rise to a Mori fibration structure for $X^*$;
    \item if $l=0$ then only one of the following occurs:
    \begin{itemize}
      \item either $D$ is nef and $X=X^*$ is already a $D$-minimal model,
      \item or $X$ admits a $D$-negative contraction over a lower dimensional variety, giving rise to a Mori fibration structure for $X$.
    \end{itemize}
    \end{enumerate}
\end{definition}

We are then in a position to state the following

\begin{theorem}\label{thm:MMP}
  Let $X$ be a complete wMDS and consider a divisor $D\in \Div(X)$. Then there exists a MMP for $D$ and
  \begin{enumerate}
    \item $X^*$ is a $D$-minimal model, if and only if $[D]\in\overline{\Eff}(X)$; in this case $D^*$ is \emph{semiample} (i.e. $|mD^*|$ is base-point-free for $m\gg 0$);
    \item $X^*$ is a \sqm of $X$ if and only if $[D]\in\overline{\Mov}(X)$.
  \end{enumerate}
  This $D$-MMP is not unique and, if $D$ is not nef, the terminal model $X^*$ can always be assumed to be a MDS.
\end{theorem}

\begin{proof} If $D$ is nef, then $X=X^*$ is already a minimal model. Then one can assume $D$ is not nef, that is $[D]\not\in\g=\Nef(X)$. Apply Lemma~\ref{lem:sQm}. Then one has a \sqm $f:X\dashrightarrow X'$ such that $X'$ is a MDS.
Let $D':=f_*D$ be the birational transform of $D$.
If $D'$ is nef, then $f|_X:X\dashrightarrow X'$ is a $D$-MMP and the MDS $X'$ is a $D$-minimal model of $X$. In fact, $f$ turns out to be $D$-negative as $f$ contracts precisely all the curves whose classes are in the face $\g^\vee\preceq\overline{NE}(X):=\Nef(X)^\vee$, which is the dual face of $\g\preceq\g':=\Nef(X')$\,. Every such curve is $D$-negative since $[D]\not\in \g$.

Then assume $D'$ is not nef and run a $D'$-MMP for $X'$, which exists by \cite[Prop.~1.11\,(1)]{Hu-Keel}. This gives a $D$-MMP with terminal model $X^*$ which is a MDS.

In particular $D^*$ is nef if and only if it is semiample. By \cite[Prop.~1.11\,(2)]{Hu-Keel} and Theorem~\ref{thm:complete-emb}\,(3), the birational map $g:X'\dashrightarrow X^*$ is induced by an inclusion $$\Nef(X^*)\stackrel{g^*}{\hookrightarrow}\overline{\Eff}(X')\stackrel{f^*}{\cong}\overline{\Eff}(X)$$
and it is a \sqm if and only if
$$\Nef(X^*)\stackrel{g^*}{\hookrightarrow}\overline{\Mov}(X')\stackrel{f^*}{\cong}\overline{\Mov}(X)\,.$$
The first inclusion proves that if $X^*$ is a $D$-minimal model then $[D]\in\overline{\Eff}(X)$ and the second one proves that if $X^*$ is also a \sqm of $X$ then $[D]\in\overline{\Mov}(X)$.

For the converse observe that if $[D]\in\overline{\Eff}(X)$ then there exists a g-chamber $\g^*\subseteq\overline{\Eff}(X)$ such that $[D]\in\g^*$. The choice of $\g^*$ determines a projective toric variety $Z^*(\g^*,\Si_{\g^*})$. On the other hand, $X'$ is projective, implying that $\g'=\Nef(X')$ is a g-chamber, whose choice determines the projective toric variety $Z(\g',\Si_{\g'})$ and a birational map $G:Z\dashrightarrow Z^*$. Then the birational transform $X^*:=G_*(X')$ gives a $D$-minimal model of $X$, so ending up the proof of (1). Moreover, if $[D]\in\overline{\Mov}(X)$ then we can choose $\g^*\subseteq\overline{\Mov}(X)$, meaning that $g:=G|_{X'}:X'\dashrightarrow X^*$ is a \sqm and proving (2).
\end{proof}

\subsection{Rational contraction of a complete wMDS}\label{ssez:contrraz}
 Here the goal is proposing an extension, to the non-projective setup, of Hu-Keel results \cite[Prop.~1.11\,(3) and Prop.~2.11\,(4)]{Hu-Keel}. The key ingredient is still Lemma~\ref{lem:sQm}.

 Recalling \cite[Def.~1.0]{Hu-Keel} and \cite[Def.~2.1]{CasagrandeMDS} let us give the following

\begin{definition}[Rational contraction]\label{def:rat-contr}
A rational dominant map $f:X\dashrightarrow Y$ from a $\Q$-factorial and complete algebraic variety $X$ to a normal and complete algebraic variety $Y$ is called a \emph{rational contraction} if there exists a resolution of $f$, that is a morphism
$$\xymatrix{\phi:\widehat{\Ga}\ar[r]& \overline{\Ga}(f)\subseteq X\times Y}$$
on the closure of the graph $\Ga(f)$ of $f$, such that:
\begin{enumerate}
\item $\widehat{\Ga}$ is a smooth and complete algebraic variety,
\item the following diagram commutes
\begin{equation*}
  \xymatrix{\widehat{\Ga}\ar[ddr]_-\mu\ar[drr]^-{f'}\ar[dr]^>>>\phi&&\\
            &X\times Y\ar[d]^-p\ar[r]^q&Y\\
            &X\ar@{-->}[ur]^-f&}
\end{equation*}
  \item $\mu$ is a birational morphism,
  \item $f'$ has connected fibers,
  \item every $\mu$-exceptional divisor $E\subset\widehat{\Ga}$ is an $f'$-exceptional divisor.
\end{enumerate}
In particular, if $f$ is birational then it is called a \emph{birational contraction}. The latter happens if and only if $f'$ is a birational morphism. Otherwise $\dim Y<\dim X$ and $f$ is called \emph{of fiber type}.
\end{definition}

\begin{remark}
  Let $f:X\dashrightarrow Y$ be a birational map, with $X$ $\Q$-factorial and complete and $Y$ normal and complete. Then $f$ is a birational contraction if and only if there exist open Zariski subsets $U\subseteq X$ and $V\subseteq Y$ such that $f|_U$ is an isomorphism between $U$ and $V$ and $\codim(Y\setminus V)\geq 2$ (this fact follows immediately by item (5) in the Definition~\ref{def:rat-contr}: see \cite[Rem.~2.2]{CasagrandeMDS} for the details).

  Consequently, a birational map $g:X\dashrightarrow X'$ between $\Q$-factorial and complete algebraic varieties is a \sqm if and only both $g$ and $g^{-1}$ are birational contractions.
\end{remark}

\begin{theorem}
Let $X$ be a complete wMDS with Cox basis $\X$ and $i:X\hookrightarrow W$ its $\X$-canonical toric embedding. Then:
\begin{enumerate}
\item up to isomorphisms, $X$ admits a finite number $s>0$ of distinct \sqm
$$g_i:X\dashrightarrow X_i\quad ,\quad i=1,\ldots,s$$
one of which is given by the identity $id_X$;
\item $\overline{\Mov}(X)=\bigcup_{i=1}^s g_i^*(\Nef(X_i))$\,;
\item the collection of convex cones given by
$$\mathcal{M}_X:=\{\s\,|\,\s\ \text{is a face of}\ g_i^*(\Nef(X_i))\,,\ i=1,\ldots,s \}$$
is a fan whose support is $|\mathcal{M}_X|=\overline{\Mov(X)}$\,;
\item up to isomorphisms, $X$ admits a finite number of rational contractions which are in $1:1$ correspondence with the cones of the fan $\mathcal{M}_X$, via the association:
$$\xymatrix{(f:X\dashrightarrow Y)\ar@{|->}[r]& f^*\Nef(Y)}$$
\item for every rational contraction $f:X\dashrightarrow Y$ there exists a toric rational contraction $\tilde{f}:W\dashrightarrow Y$, such that $f=\tilde{f}|_X$\,.
\end{enumerate}
\end{theorem}

\begin{proof}
By Lemma \ref{lem:sQm} there exists a \sqm $g:X\dashrightarrow X'$ such that $X'$ is a MDS. In particular, if $Q$ is a weight matrix of $W$,
\begin{equation*}
 \Mov(Q)=\overline{\Mov}(W)\stackrel{i^*}{\cong}\overline{\Mov(X)}\stackrel{(g^{-1})^*}
 {\cong}\overline{\Mov}(X')\,.
\end{equation*}
Items from (1) to (5) hold for $X'$ by Prop.~2.9 and Def.~1.10\,(3), Prop.~1.11\,(2)-(3), Prop.~2.11\,(4) in \cite{Hu-Keel}. Moreover there is a bijection between rational contractions on $X'$ and rational contraction on $X$ by associating
\begin{equation*}
\xymatrix{(f':X'\dashrightarrow Y)\ar@{|->}[r]&(f'\circ g:X\dashrightarrow Y)}
\end{equation*}
Clearly such a $1:1$ correspondence sends a \sqm in a \sqm. This suffices to prove item from (1) to (4).

For (5), let $W'$ be the canonical ambient toric variety of $X'$. By \cite[Prop.~2.11\,(4)]{Hu-Keel} the \sqm $g$ is induced by a \sqm $\tilde{g}:W\dashrightarrow W'$ such that $\tilde{g}|_X=g$. Then the rational contraction
$$f=f'\circ g= (\tilde{f'}\circ\tilde{g})|_X:X\dashrightarrow Y$$
is the restriction to $X$ of the toric rational contraction $\tilde{f'}\circ\tilde{g}:W\dashrightarrow Y$\,.
\end{proof}

\section{Weak Mori Dream Spaces and weak Fano varieties}\label{sez:weakFano}
In the following an irreducible, normal and complete algebraic variety $X$ will be called \emph{(weak) $\Q$-Fano} if it is $\Q$-Gorenstein and there exists an integer $k>0$ such that $-kK_X$ is an ample (resp. nef and big) Cartier divisor (recall that a \emph{big} divisor is one admitting maximal Iitaka dimension, see e.g. \cite[Def.~2.2.1]{Lazarsfeld}). As a consequence of a well known result of Birkar, Cascini, Hacon and McKernan \cite[\S~1.3]{BCHMcK} a $\Q$-factorial $\Q$-Fano variety turns out to be a MDS. It seems then natural asking if a similar result extends to a weak $\Q$-Fano variety, implying that it is also a wMDS. Unfortunately the latter cannot be established in such a generality, but one has to impose the existence of a projective \sqm.

\begin{definition}[log pair, see \cite{KM}]\label{def:log}
    Let $X$ be an irreducible and  normal algebraic variety. An effective $\Q$-divisor $\De$ such that $K_X+\De$ is $\Q$-Cartier is called a \emph{boundary divisor} of $X$ and the pair $(X,\De)$ is called a \emph{log pair}.
  \end{definition}

   Let $(X,\De)$ be a log pair and $f:Y\longrightarrow X$ be a birational morphism from an irreducible, normal, algebraic variety $Y$. If $\De=\sum_i a_i D_i$, as a sum of prime divisors with rational coefficients, then its birational transform by $f$ is given by
   $$f_*^{-1}\De=\sum_ia_if_*^{-1}D_i$$
   where $f^{-1}_*\De,f^{-1}_*D_i$ denote the birational transforms of $\De$ and $D_i$, respectively (following notation in Koll\'{a}r-Mori \cite[(10) and (11), p.5]{KM}).
   Let $m$ be the Cartier index of $K_X+\De$ (i.e. the least positive integer $m$ such that $m(K_X+\De)$ is Cartier). Then, for every irreducible exceptional divisor $E_j\subseteq\Exc(f)$ there exists a rational number $a(E_j,X,\De)$ such that
   \begin{equation}\label{discrepanza}
     m(K_Y+f_*^{-1}\De)\sim f^*(m(K_X+\De))+m\sum_ja(E_j,X,\De)E_j\,.
   \end{equation}
Set $a(f_*^{-1}D_i,X,\De):=-a_i$ and $a(E,X,\De)=0$ for any further prime divisor not contained in $\Exc(f)\cup\Supp(f_*^{-1}\De)$. Then (\ref{discrepanza}) can be rewritten as follows:
\begin{equation}\label{discrepanza2}
  mK_Y-f^*(m(K_X+\De))\sim m\sum_{E\,\text{prime}}a(E,X,\De)E= m A(X,\De)\,.
\end{equation}
The $\Q$-divisor $A(X,\De)$ is called the \emph{discrepancy divisor} and its coefficient $a(E,X,\De)$ is called the \emph{discrepancy of $E$}, with respect to the log pair $(X,\De)$ \cite[Def.~2.25]{KM}. This is well defined, as discrepancy $a(E,X,\De)$ depends only on the valuation $v(E,Y)$ and not on the particular choice of $f$ (see \cite[Rem.~2.23]{KM}).

\begin{definition}[Kawamata log terminal (Klt) singularities]\label{def:sings} Let $(X,\De)$ be a log pair and $A(X,\De)$ its discrepancy divisor. Let us denote by $\llcorner\,\ \lrcorner$ the integral part of a $\Q$-divisor, obtained by taking the integral parts of the coefficients. Then:
\begin{enumerate}
  \item $(X,\De)$ is a called a \emph{Kawamata log terminal (Klt)} pair if $\llcorner A(X,\De)\lrcorner\geq 0$ and $\llcorner\De\lrcorner =0$,
\end{enumerate}
In particular $X$ is said admitting \emph{Klt singularities} if there exists a boundary divisor $\De$ such that $(X,\De)$ is a Klt pair.\\
  \end{definition}


\begin{definition}[log (weak) Fano, see Def.~2.5 in \cite{PS} and \cite{GOST}]\label{def:Fano} A Klt log pair $(X,\De)$ with $X$ complete is called \emph{log Fano} (resp. \emph{log weak Fano}) if a suitable multiple of $-(K_X+\De)$ is ample (resp. big and nef). Moreover $X$ is called of \emph{Fano type} if it admits a boundary divisor $\De$ such that $(X,\De)$ is a log Fano pair.
\end{definition}
Notice that if $(X,\De)$ is a log Fano pair then $X$ is necessarily a projective variety, but this is no longer true if $(X,\De)$ is a log weak Fano pair.

\begin{proposition}[Lemma-Definition 2.6 in \cite{PS}]\label{prop:wlog->log}
  If $X$ is a normal projective variety admitting a boundary divisor $\De$ such that $(X,\De)$ is a log weak Fano pair then $X$ is of Fano type.
\end{proposition}

\subsection{Log weak Fano vs wMDS} As already mentioned above, \cite[Cor.~1.3.2]{BCHMcK} gua\-ran\-tees that
a $\Q$-factorial variety of Fano type is a MDS. This result has been later improved by the following result

\begin{theorem}[Theorem 1.1 in \cite{GOST}]\label{thm:FT<->MDS} A $\Q$-factorial variety $X$ is of Fano type if and only $X$ is a MDS and the total coordinate space $\overline{X}$ has at worst Klt singularities.
  \end{theorem}

  This result, joint with the previous Proposition~\ref{prop:wlog->log}, gives the following

\begin{corollary}\label{cor:proj&wF->MDS}
  A $\Q$-factorial and projective variety $X$ admitting a boundary divisor $\De$ such that $(X,\De)$ is a log weak Fano pair is a MDS whose total space $\overline{X}$ has at worst Klt singularities. In particular a $\Q$-factorial and projective weak $\Q$-Fano variety is a MDS whose total space has at worst Klt singularities.
\end{corollary}

\begin{corollary}\label{cor:wMDS}
   Let $X$ be a $\Q$-factorial and complete algebraic variety admitting a projective small $\Q$-factorial modification, i.e.
   \begin{equation}\label{proj-sQm}
     \exists\,f,N:\quad\xymatrix{X\ar@{-->}[r]^-f_{\text{\sqm}}&Y\ar@{^(->}[r]&\P^N}\,.
   \end{equation}
   Then the following are equivalent:
   \begin{enumerate}
     \item $X$ is a wMDS and the total space $\overline{X}$ has at worst Klt singularities,
     \item $Y$ is a MDS and the total space $\overline{Y}$ has at worst Klt singularities,
     \item $Y$ is of Fano type,
     \item there exists a boundary divisor $\De$ such that $(Y,\De)$ is a log weak Fano pair.
   \end{enumerate}
    \end{corollary}

\begin{proof}
  (1)$\Rightarrow$(2): The \sqm $f$ induces isomorphisms $\Cl(X)\cong\Cl(Y)$ and $\Cox(X)\cong\Cox(Y)$, as it is an isomorphism in codimension 1. Then $Y$ is a projective wMDS, that is a MDS. Then $\overline{X}=\Spec(\Cox(X))\cong\Spec(\Cox(Y))=\overline{Y}$ have the same type of singularities. In particular $\overline{Y}\cong\overline{X}$ admits at worst Klt singularities.

  (2)$\Rightarrow$(3): It is the necessary condition in Theorem~\ref{thm:FT<->MDS}.

  (3)$\Rightarrow$(4): This is obvious by Definition~\ref{def:Fano}.

  (4)$\Rightarrow$(1): By Proposition~\ref{prop:wlog->log}, $Y$ is of Fano type. The sufficient condition in Theorem~\ref{thm:FT<->MDS} implies that $Y$ is a MDS and $\overline{Y}$ has at worst Klt singularities. Finally by same argument applied when proving (1)$\Rightarrow$(2), which is  the necessary condition in Lemma~\ref{lem:sQm},  $X$ turns out to be a wMDS: in particular $\overline{X}$ admits at worst Klt singularities.
\end{proof}

\begin{remark}
  If one of the equivalent conditions listed in Corollary~\ref{cor:wMDS} holds then  hypothesis (\ref{proj-sQm}) can be obtained by asking that
  \begin{equation}\label{mov}
    \exists\,m\in\N\,:\, -mK_X\ \text{is a movable Cartier divisor.}
  \end{equation}

  \noindent In fact, by assuming item (1), there certainly exists a chamber $\g\subseteq\Mov(X)$ such that $[-K_X]\in\g$. By the same argument proving the sufficient condition in Lemma~\ref{lem:sQm}, the choice of $\g$ uniquely determines a MDS $Y$ with a \sqm $f:X\dashrightarrow Y$ such that $\g\subseteq f^*\Nef(Y)$. This is enough to get (\ref{proj-sQm}).
  Moreover, this means that $$[-K_X]=f^*[-K_Y]\in\g\subseteq f^*\Nef(Y)\subseteq f^*\Mov(Y)=\Mov(X)\,.$$
  In particular $-mK_Y$ is a nef Cartier divisor (we are actually running a MMP for $-K_X$). If we now further strengthen (\ref{mov}) to ask that
  \begin{equation}\label{big&mov}
    \exists\,m\in\N\,:\, -mK_X\ \text{is a big and movable Cartier divisor}
  \end{equation}
  then $-mK_Y$ is big and $Y$ turns out to be a projective weak Fano variety. Conversely, if $Y$ in (\ref{proj-sQm}) is assumed to be a weak $\Q$-Fano variety, then Corollary~\ref{cor:proj&wF->MDS} gives that $Y$ is a MDS whose total space $\overline{Y}$ has at worst Klt singularities. This gives item (2) in Corollary~\ref{cor:wMDS} so implying item (1) and (\ref{big&mov}).
  \end{remark}

  This discussion shows that Corollary~\ref{cor:wMDS} can be thought of as a generalization of \cite[Lemma~4.10]{McK}. In particular it proves that

  \begin{proposition}\label{prop:wFsQm}
    The following assertions are equivalent:
    \begin{enumerate}
      \item $X$ is a wMDS whose total space $\overline{X}$ admits at worst Klt singularities and such that $-K_X$ is a big and movable $\Q$-divisor,
      \item $X$ is a $\Q$-factorial and complete algebraic variety admitting a projective weak $\Q$-Fano \sqm $Y$.
    \end{enumerate}
  \end{proposition}

  \begin{acknowledgements}
First of all I would like to thank Lea Terracini, with whom  several parts of  the present  work have been discussed. Her patience, criticism and encouragement have been an invaluable help in realizing this paper.

\noindent Moreover, I would like to thank Cinzia Casagrande for pointing me out her precious notes \cite{CasagrandeMDS} and for some useful remark.

\noindent Finally, many thanks are due to Antonio Laface for several fruitful hints.
\end{acknowledgements}

\end{document}